%% file: dquer_main.tex
\documentclass[a4paper, intlimits, reqno]{amsart}

\usepackage[english]{babel}
\usepackage[T1]{fontenc}
\usepackage[utf8]{inputenc}

\usepackage{amsmath}
\usepackage{amssymb}
\usepackage{MnSymbol}
\usepackage{amsthm}
\allowdisplaybreaks
\usepackage{amsfonts}
\usepackage{mathrsfs} 
\usepackage{mathtools}
\usepackage{nicefrac}
\usepackage{enumitem}
\usepackage{multicol}
\usepackage{xurl}
\usepackage{dsfont}
\usepackage[numbers]{natbib}
\usepackage{prettyref}
\usepackage{doi}

\newrefformat{def}{Definition \ref{#1}}
\newrefformat{rem}{Remark \ref{#1}}
\newrefformat{sect}{Section \ref{#1}}
\newrefformat{prop}{Proposition \ref{#1}}
\newrefformat{thm}{Theorem \ref{#1}}
\newrefformat{cor}{Corollary \ref{#1}}
\newrefformat{ex}{Example \ref{#1}}
\newrefformat{cond}{Condition \ref{#1}}
\newrefformat{conv}{Convention \ref{#1}}

\swapnumbers
\newtheoremstyle{dotless}{}{}{\itshape}{}{\bfseries}{}{}{}
\theoremstyle{dotless}
\theoremstyle{plain}
\newtheorem{thm}{Theorem}[section]
\newtheorem{lem}[thm]{Lemma}

\newtheorem{cor}[thm]{Corollary}
\theoremstyle{definition}
\newtheorem{defn}[thm]{Definition}
\newtheorem{rem}[thm]{Remark}
\newtheorem{exa}[thm]{Example}
\newtheorem{conv}[thm]{Convention}
\newtheorem*{condPN}{Condition (PN)}{\bf}{\rm}
\newtheorem*{condWR}{Condition (WR)}{\bf}{\rm}

\newcommand{\N} {\mathbb{N}}
\newcommand{\Z} {\mathbb{Z}}

\newcommand{\R} {\mathbb{R}}
\newcommand{\C} {\mathbb{C}}
\newcommand{\K} {\mathbb{K}}

\DeclareMathOperator{\re}{Re}
\DeclareMathOperator{\im}{Im}

\providecommand{\differential}{\mathrm{d}}
\renewcommand{\d}{\differential}

\begin{document}

\title[Surjectivity]{Surjectivity of the $\overline{\partial}$-operator between weighted spaces of smooth vector-valued functions}
\author[K.~Kruse]{Karsten Kruse}
\address{Hamburg University of Technology\\ Institute of Mathematics \\
Am Schwarzenberg-Campus~3 \\
21073 Hamburg \\
Germany}
\email{karsten.kruse@tuhh.de}



\date{\today}
\begin{abstract}
We derive sufficient conditions for the surjectivity of the Cauchy-Riemann operator $\overline{\partial}$ 
between weighted spaces of smooth Fr\'echet-valued functions. This is done by establishing an analog of H\"ormander's theorem on 
the solvability of the inhomogeneous Cauchy-Riemann equation in a space of smooth $\C$-valued functions whose topology
is given by a whole family of weights. Our proof relies on a weakened variant of weak reducibility of the corresponding subspace 
of holomorphic functions in combination with the Mittag-Leffler procedure. Using tensor products, we deduce the corresponding result 
on the solvability of the inhomogeneous Cauchy-Riemann equation for Fr\'echet-valued functions.
\end{abstract}
\maketitle
\section{Introduction}
\input{Intro}

\section{Notation and Preliminaries}
\input{Notation}
\section{From \texorpdfstring{$\sup$}{sup}- to \texorpdfstring{$L^{q}$}{Lq}-seminorms}
\input{SUPtoL2}
\section{Main result}
\input{main_thm}
\section{Very weak reducibility and applications of the main result}
\input{density3}
\bibliography{biblio}
\bibliographystyle{plainnat}
\end{document}

%% file: Intro.tex
We study the Cauchy-Riemann operator between weighted spaces of smooth functions with values in a Fr\'echet space. 
Let $E$ be a complete locally convex Hausdorff space over $\C$, $\Omega\subset\R^{2}$ open and 
$\mathcal{E}(\Omega):=\mathcal{C}^{\infty}(\Omega,\C)$ the space of infinitely continuously partially 
differentiable functions from $\Omega$ to $\C$.
It is well-known that the Cauchy-Riemann operator 
\[
 \overline{\partial}:=\frac{1}{2}(\partial_{1}+i\partial_{2})\colon \mathcal{E}(\Omega)\to \mathcal{E}(\Omega)
\]
is surjective (see e.g.\ \cite[Theorem 1.4.4, p.\ 12]{H3}). Since $\mathcal{E}(\Omega)$, 
equipped with the usual topology of uniform convergence of partial derivatives of any order on compact subsets,
is a nuclear Fr\'{e}chet space by \cite[Example 28.9 (1), p.\ 349]{meisevogt1997}, 
we have the topological isomorphy $\mathcal{E}(\Omega,E)\cong\mathcal{E}(\Omega)\widehat{\otimes}_{\pi}E$ by 
\cite[Theorem 44.1, p.\ 449]{Treves} where $\mathcal{E}(\Omega)\widehat{\otimes}_{\pi}E$ is the 
completion of the projective tensor product. Due to classical theory of tensor products, 
the surjectivity of $\overline{\partial}$ implies the surjectivity of 
\[
 \overline{\partial}^{E}\colon \mathcal{E}(\Omega,E)\to \mathcal{E}(\Omega,E)
\]
for Fr\'echet spaces $E$ over $\C$ (see e.g.\ \cite[Satz 10.24, p.\ 255]{Kaballo}) 
where $\mathcal{E}(\Omega,E)$ is the space of infinitely continuously partially differentiable functions 
from $\Omega$ to $E$ and $\overline{\partial}^{E}$ is the Cauchy-Riemann operator for $E$-valued functions. 
In other words, given $f\in\mathcal{E}(\Omega,E)$ there is a solution $u\in\mathcal{E}(\Omega,E)$ of the 
$\overline{\partial}$-problem, i.e.\ 
\begin{equation}\label{eq:intro.0}
 \overline{\partial}^{E}u=f.
\end{equation}
Now, we consider the following situation. Denote by $(p_{\alpha})_{\alpha\in\mathfrak{A}}$ a system 
of seminorms inducing the locally convex Hausdorff topology of $E$. 
Let $f$ fulfil some additional growth conditions given by an increasing family of positive continuous functions 
$\mathcal{V}:=(\nu_{n})_{n\in\N}$ on an increasing sequence of open subsets $(\Omega_{n})_{n\in\N}$ 
of $\Omega$ with $\Omega=\bigcup_{n\in\N}\Omega_{n}$, namely,  
\[
|f|_{n,m,\alpha}:=\sup_{\substack{x\in \Omega_{n}\\ \beta\in\N^{2}_{0},\,|\beta|\leq m}}
p_{\alpha}\bigl((\partial^{\beta})^{E}f(x)\bigr)\nu_{n}(x)<\infty
\]
for every $n\in\N$, $m\in\N_{0}$ and $\alpha\in\mathfrak{A}$. 
Let us call the space of smooth functions having this growth $\mathcal{EV}(\Omega,E)$. 
Then there is always a solution $u\in\mathcal{E}(\Omega,E)$ of \eqref{eq:intro.0}.
Our aim is to derive sufficient conditions such that there is a solution $u$ of \eqref{eq:intro.0} 
having the same growth as the right-hand side $f$. So we are interested under which conditions the Cauchy-Riemann operator
\[
 \overline{\partial}^{E}\colon \mathcal{EV}(\Omega,E)\to \mathcal{EV}(\Omega,E)
\]
is surjective. The interest in solving the vector-valued $\overline{\partial}$-problem arises in \cite{ich} from 
the construction of smooth functions with exponential growth on strips (with holes), which is used to prove the 
flabbyness of the sheaf of vector-valued Fourier hyperfunctions, see \cite[6.8 Lemma, p.\ 118]{ich} 
and \cite[6.11 Theorem, p.\ 136]{ich}. However, the interest in the vector-valued $\overline{\partial}$-problem 
may also be motivated from the scalar-valued problem, namely, from the question of parameter dependence. 
If e.g.\ the right-hand side $f_{\lambda}\in\mathcal{EV}(\Omega)$ depends continuously on a parameter $\lambda\in[0,1]$, 
then there are solutions $u_{\lambda}\in\mathcal{EV}(\Omega)$ of $\overline{\partial}u_{\lambda}=f_{\lambda}$ which 
depend continuously on $\lambda$ as well if the vector-valued $\overline{\partial}$-problem \eqref{eq:intro.0} is 
solvable for the Banach space $E=\mathcal{C}([0,1],\C)$ of continuous $\C$-valued functions on $[0,1]$.

The difficult part is to solve the $\overline{\partial}$-problem in the scalar-valued case, i.e.\ in $\mathcal{EV}(\Omega)$. 
In the case that $\mathcal{V}=(\nu)$ and $\Omega_{n}=\Omega$ for all $n\in\N$ 
where $\nu$ is a weight which permits growth near infinity there is 
a classical result by H\"ormander \cite[Theorem 4.4.2, p.\ 94]{H3} on the solvability of the $\overline{\partial}$-problem 
(in the distributional sense) in weighted spaces of $\C$-valued square-integrable functions of the form
\[
L^{2}\nu(\Omega):=\{f\colon\Omega\to\C\;\text{measurable}\;|\;\int_{\Omega}|f(z)|^{2}\nu(z)\d z<\infty\}.
\]
The opposite situation where the weight $\nu$ permits decay near infinity is handled in \cite[Theorem 1.2, p.\ 351]{Hedenmalm2015} and more general in \cite{Amar2016}.
The solvability of the $\overline{\partial}$-problem in weighted $L^{2}$-spaces and its subspaces of holomorphic functions 
has some nice applications (see \cite{Hoermander2003}) and the properties of the canonical solution operator 
to $\overline{\partial}$ are subject of intense studies \cite{Bonami1990}, \cite{Charpentier2014}, \cite{Haslinger2001}, 
\cite{Haslinger2002}, \cite{Haslinger2007}.

If there is a whole system of weights $\mathcal{V}=(\nu_{n})_{n\in\N}$, i.e.\ the $\overline{\partial}$-problem 
is considered in the projective limit spaces $L^{2}\mathcal{V}(\Omega):=\bigcap_{n\in\N}L^{2}\nu_{n}(\Omega_{n})$ 
or $L^{\infty}\mathcal{V}(\Omega):=\bigcap_{n\in\N}L^{\infty}\nu_{n}(\Omega)$ where
\[
L^{\infty}\nu_{n}(\Omega):=\{f\colon\Omega\to\C\;\text{measurable}\;|\;\sup_{z\in\Omega}|f(z)|\nu_{n}(z)<\infty\},
\]
then solving the $\overline{\partial}$-problem becomes more complicated 
since a whole family of $L^{2}$- resp.\ $L^{\infty}$-estimates has to be satisfied. 
Such a $\overline{\partial}$-problem is usually solved by a combination of H\"ormander's classical result 
with the Mittag-Leffler procedure. However, 
this requires the projective limit $\mathcal{O}^{2}\mathcal{V}(\Omega):=\bigcap_{n\in\N}\mathcal{O}^{2}\nu_{n}(\Omega)$ 
resp.\ $\mathcal{O}^{\infty}\mathcal{V}(\Omega):=\bigcap_{n\in\N}\mathcal{O}^{\infty}\nu_{n}(\Omega)$, where
\[
\mathcal{O}^{k}\nu_{n}(\Omega):=\{f\in L^{k}\nu_{n}(\Omega)\;|\;f\;\text{holomorphic}\},
\]
to be weakly reduced, i.e.\ for every $n\in\N$ there is $m\in\N$ such that $\mathcal{O}^{k}\mathcal{V}(\Omega)$ is 
dense in $\mathcal{O}^{k}\nu_{m}(\Omega)$ with respect to the topology of $\mathcal{O}^{k}\nu_{n}(\Omega)$ 
for $k=2$ resp.\ $k=\infty$, see \cite[Theorem 3, p.\ 56]{Epifanov1992}, \cite[1.3 Lemma, p.\ 418]{Langenbruch1994} 
and \cite[Theorem 1, p.\ 145]{Polyakova2017}. Unfortunately, the weak reducibility of the projective limit is not easy to check. 
Furthermore, in our setting we have to control the growth of the partial derivatives as well and the sequence $(\Omega_{n})_{n\in\N}$ 
usually consists of more than one set.  

Let us outline our strategy to solve the $\overline{\partial}$-problem in $\mathcal{EV}(\Omega,E)$ for Fr\'echet spaces $E$ over $\C$.
In Section 2 we fix the notation and state some preliminaries. 
In Section 3 we phrase sufficient conditions (see Condition $(PN)$) 
such that there is an equivalent system of $L^{q}$-seminorms on $\mathcal{EV}(\Omega)$ (see \prettyref{lem:switch_top}). 
If they are fulfilled for $q=1$, then $\mathcal{EV}(\Omega)$ is a nuclear Fr\'echet space 
by \cite[Theorem 3.1, p.\ 188]{kruse2018_4}. If they are fulfilled for $q=2$ as well, 
we can use H\"ormander's $L^{2}$-machinery and the hypoellipticity of $\overline{\partial}$ 
to solve the scalar-valued equation \eqref{eq:intro.0}
on each $\Omega_{n}$ with given $f\in\mathcal{EV}(\Omega)$ and  
a solution $u_{n}\in\mathcal{E}(\Omega_{n})$ satisfying $|u_{n}|_{n,m}<\infty$ for every $m\in\N_{0}$. 
In Section 4 the solution $u\in\mathcal{EV}(\Omega)$ is then constructed 
from the $u_{n}$ by using the Mittag-Leffler procedure in our main \prettyref{thm:scalar_CR_surjective}, 
which requires a density condition on the kernel of $\overline{\partial}$. 
Due to \cite[Example 16 c), p.\ 1526]{kruse2017} we have $\mathcal{EV}(\Omega,E)\cong\mathcal{EV}(\Omega)\widehat{\otimes}_{\pi}E$ 
if Condition $(PN)$ holds for $q=1$ and are able to
lift the surjectivity from the scalar-valued to the Fr\'echet-valued case in \prettyref{cor:frechet_CR_surjective}. 
This density condition can be regarded 
as a weakened variant of weak reducibility of the subspace of $\mathcal{EV}(\Omega)$ consisting of holomorphic functions.
In our last section we state sufficient conditions on $\mathcal{V}$ and $(\Omega_{n})_{n\in\N}$ in \prettyref{thm:dense_proj_lim} for our density condition 
to hold that are more likely to be checked. Further, we give examples of weights $\mathcal{V}$ and sets 
$(\Omega_{n})_{n\in\N}$ that satisfy our conditions in \prettyref{ex:families_of_weights_2}.
The stated results are obtained by generalising the methods in \cite[Chap.\ 5]{ich} where 
the special case $\nu_{n}(z):=\exp(-|\re(z)|/n)$ and, amongst others, 
$\Omega_{n}:=\{z\in\C\;|\; 1/n<|\im(z)|<n\}$ is treated 
(see \cite[5.16 Theorem, p.\ 80]{ich} and \cite[5.17 Theorem, p.\ 82]{ich}).

%% file: Notation.tex
We define the distance of two subsets $M_{0}, M_{1} \subset\R^{d}$, $d\in \N$, w.r.t.\ a norm $\|\cdot\|$ on $\R^{d}$ via
\[
  \d^{\|\cdot\|}(M_{0},M_{1}) 
:=\begin{cases}
   \inf_{x\in M_{0},\,y\in M_{1}}\|x-y\| &,\;  M_{0},\,M_{1} \neq \emptyset, \\
   \infty &,\;  M_{0}= \emptyset \;\text{or}\; M_{1}=\emptyset.
  \end{cases}
\]
Moreover, we denote by $\|\cdot\|_{\infty}$ the sup-norm, by $|\cdot|$ the Euclidean norm, by 
$\langle\cdot|\cdot\rangle$ the usual scalar product on $\R^{d}$ and by
$\mathbb{B}_{r}(x):=\{w\in\R^{d}\;|\;|w-x|<r\}$ the Euclidean ball around $x\in\R^{d}$ with radius $r>0$.
We denote the complement of a subset $M\subset \R^{d}$ by $M^{C}:= \R^{d}\setminus M$,
the set of inner points of $M$ by $\mathring{M}$, the closure of $M$ by $\overline{M}$ and the boundary of $M$ by $\partial M$.
Further, we also use for $z=(z_{1},z_{2})\in\R^{2}$ a notation of mixed-type
\[
z=z_{1}+iz_{2}
 =(z_{1},z_{2})
 =\begin{pmatrix}
  z_{1}\\
  z_{2}
  \end{pmatrix},
\]
hence identify $\R^{2}$ and $\C$ as (normed) vector spaces.
For a function $f\colon M\to\C$ and $K\subset M$ we denote by $f_{\mid K}$ the restriction of $f$ to $K$ and by 
\[
 \|f\|_{K}:=\sup_{x\in K}|f(x)|
\]
the sup-norm on $K$.

By $E$ we always denote a non-trivial locally convex Hausdorff space over the field 
$\K=\R$ or $\C$ equipped with a directed fundamental system of 
seminorms $(p_{\alpha})_{\alpha\in \mathfrak{A}}$. 
If $E=\K$, then we set $(p_{\alpha})_{\alpha\in \mathfrak{A}}:=\{|\cdot|\}$. 
Further, we denote by $L(F,E)$ the space of continuous linear maps from 
a locally convex Hausdorff space $F$ to $E$. If $E=\K$, we write $F':=L(F,\K)$ for the dual space of $F$.

We recall the following well-known definitions concerning continuous partial differentiability of 
vector-valued functions (c.f.\ \cite[p.\ 237]{kruse2018_2}). A function $f\colon\Omega\to E$ on an open set 
$\Omega\subset\mathbb{R}^{d}$ to $E$ is called continuously partially differentiable ($f$ is $\mathcal{C}^{1}$) 
if for the $n$-th unit vector $e_{n}\in\mathbb{R}^{d}$ the limit
\[
(\partial^{e_{n}})^{E}f(x):=(\partial_{x_{n}})^{E}f(x):=(\partial_{n})^{E}f(x)
:=\lim_{\substack{h\to 0\\ h\in\mathbb{R}, h\neq 0}}\frac{f(x+he_{n})-f(x)}{h}
\]
exists in $E$ for every $x\in\Omega$ and $(\partial^{e_{n}})^{E}f$ 
is continuous on $\Omega$ ($(\partial^{e_{n}})^{E}f$ is $\mathcal{C}^{0}$) for every $1\leq n\leq d$.
For $k\in\mathbb{N}$ a function $f$ is said to be $k$-times continuously partially differentiable 
($f$ is $\mathcal{C}^{k}$) if $f$ is $\mathcal{C}^{1}$ and all its first partial derivatives are $\mathcal{C}^{k-1}$.
A function $f$ is called infinitely continuously partially differentiable ($f$ is $\mathcal{C}^{\infty}$) 
if $f$ is $\mathcal{C}^{k}$ for every $k\in\mathbb{N}$.
The linear space of all functions $f\colon\Omega\to E$ which are $\mathcal{C}^{\infty}$ 
is denoted by $\mathcal{C}^{\infty}(\Omega,E)$. 
Let $f\in\mathcal{C}^{\infty}(\Omega,E)$. For $\beta=(\beta_{n})\in\mathbb{N}_{0}^{d}$ we set 
$(\partial^{\beta_{n}})^{E}f:=f$ if $\beta_{n}=0$, and
\[
(\partial^{\beta_{n}})^{E}f
:=\underbrace{(\partial^{e_{n}})^{E}\cdots(\partial^{e_{n}})^{E}}_{\beta_{n}\text{-times}}f
\]
if $\beta_{n}\neq 0$ as well as 
\[
(\partial^{\beta})^{E}f
:=(\partial^{\beta_{1}})^{E}\cdots(\partial^{\beta_{d}})^{E}f.
\]
Due to the vector-valued version of Schwarz' theorem $(\partial^{\beta})^{E}f$ is independent of the order of the partial 
derivatives on the right-hand side, we call $|\beta|:=\sum_{n=1}^{d}\beta_{n}$ the order of differentiation 
and write $\partial^{\beta}f:=(\partial^{\beta})^{\K}f$. Now, the precise definition 
of the weighted spaces of smooth vector-valued functions from the introduction reads as follows. 

\begin{defn}[{\cite[Definition 3.2, p.\ 238]{kruse2018_2}}]\label{def:smooth_weighted_space}
Let $\Omega\subset\R^{d}$ be open and $(\Omega_{n})_{n\in\N}$ a family of non-empty
open sets such that $\Omega_{n}\subset\Omega_{n+1}$ and $\Omega=\bigcup_{n\in\N} \Omega_{n}$.
Let $\mathcal{V}:=(\nu_{n})_{n\in\N}$ be a countable family of positive continuous functions 
$\nu_{n}\colon \Omega \to (0,\infty)$ such that $\nu_{n}\leq\nu_{n+1}$ for all $n\in\N$.
We call $\mathcal{V}$ a (directed) family of continuous weights on $\Omega$ and set
\[
\mathcal{E}\nu_{n}(\Omega_{n}, E):= \{ f \in \mathcal{C}^{\infty}(\Omega_{n}, E)\; | \;
\forall\;\alpha\in\mathfrak{A},\,m \in \N_{0}^{d}:\; |f|_{n,m,\alpha} < \infty \}
\]
for $n\in\N$ and 
\[
\mathcal{EV}(\Omega, E):=\{ f\in \mathcal{C}^{\infty}(\Omega, E)\; | \;\forall\; n \in \N:
\; f_{\mid\Omega_{n}}\in \mathcal{E}\nu_{n}(\Omega_{n}, E)\}
\]
where
\[
|f|_{n,m,\alpha}:=\sup_{\substack{x \in \Omega_{n}\\ \beta \in \N_{0}^{d}, \, |\beta| \leq m}}
p_{\alpha}\bigl((\partial^{\beta})^{E}f(x)\bigr)\nu_{n}(x).
\]
The subscript $\alpha$ in the notation of the seminorms is omitted in the scalar-valued case. 
The notation for the spaces in the scalar-valued case is 
$\mathcal{E}\nu_{n}(\Omega_{n}):=\mathcal{E}\nu_{n}(\Omega_{n},\K)$ 
and $\mathcal{EV}(\Omega):=\mathcal{EV}(\Omega,\K)$.
\end{defn}

The space $\mathcal{EV}(\Omega,E)$ is a projective limit, namely, we have
\[
\mathcal{EV}(\Omega, E)\cong \lim_{\substack{\longleftarrow\\n\in \N}}\mathcal{E}\nu_{n}(\Omega_{n}, E)
\]
where the spectral maps are given by the restrictions
\[
\pi_{k,n}\colon \mathcal{E}\nu_{k}(\Omega_{k}, E)\to \mathcal{E}\nu_{n}(\Omega_{n}, E),\;
f\mapsto f_{\mid\Omega_{n}},\;k\geq n.
\]
The space of scalar-valued infinitely differentiable functions with compact support in an open set $\Omega\subset\R^{d}$ 
is defined by the inductive limit
\[
\mathcal{D}(\Omega):=\lim_{\substack{\longrightarrow\\K\subset \Omega\;\text{compact}}}\,\mathcal{C}^{\infty}_{c}(K)
\]
where
\[
\mathcal{C}^{\infty}_{c}(K):=\{f\in\mathcal{C}^{\infty}(\R^{d},\K)\;|\;\forall \;x\notin K:\;f(x)=0\}.
\]
Every element $f$ of $\mathcal{D}(\Omega)$ can be regarded as an element of $\mathcal{D}(\R^{d})$ 
just by setting $f:=0$ on $\Omega^{C}$ and we write $\operatorname{supp} f$ for the support of $f$.
Moreover, we set for $m\in\N_{0}$ and $f\in\mathcal{D}(\R^{d})$
\[
\|f\|_{m}:=\sup_{\substack{x\in\R^{d}\\ \alpha\in\N^{d}_{0},\;|\alpha|\leq m}}|\partial^{\alpha}f(x)|.
\]
By $L^{1}(\Omega)$ we denote the space of (equivalence classes of) $\K$-valued Lebesgue integrable functions on $\Omega$,
by $L^{q}(\Omega)$, $q\in\N$, the space of functions $f$ such that $f^{q}\in L^{1}(\Omega)$ and 
by $L^{q}_{loc}(\Omega)$ the corresponding space of locally integrable functions. 
For a locally integrable function $f\in L^{1}_{loc}(\Omega)$ we denote by $T_{f}\in\mathcal{D}'(\Omega):=\mathcal{D}(\Omega)'$ 
the regular distribution defined by
\[
T_{f}(\varphi):=\int_{\R^{d}} {f(x)\varphi(x)\d x},\quad \varphi\in \mathcal{D}(\Omega).
\]
For $\alpha\in\N^{d}_{0}$ the partial derivatives of a distribution $T\in\mathcal{D}'(\Omega)$ are defined by
\[
\partial^{\alpha}T(\varphi):=\langle\partial^{\alpha}T,\varphi\rangle
:=(-1)^{|\alpha|}T(\partial^{\alpha}\varphi),\quad\varphi\in\mathcal{D}(\Omega).
\]
The convolution $T\ast\varphi$ of a distribution $T\in\mathcal{D}'(\R^{d})$ and a test function 
$\varphi\in \mathcal{D}(\R^{d})$ is defined by
\[
(T\ast\varphi)(x):=T(\varphi(x-\cdot)),\quad x\in\R^{d}.
\]
In particular, we have $\delta\ast\varphi=\varphi$ for the Dirac distribution $\delta$ and
\begin{equation}\label{distr.falt.}
(T_{f}\ast\varphi)(x)=\int_{\R^{d}} {f(y)\varphi(x-y)\d y},\quad x\in\R^{d},
\end{equation}
for $f\in L^{1}_{loc}(\R^{d})$ and $\varphi\in \mathcal{D}(\R^{d})$.
Furthermore, $\partial^{\alpha}(T\ast\varphi)=(\partial^{\alpha}T)\ast\varphi=T\ast(\partial^{\alpha}\varphi)$ 
is valid for $T\in\mathcal{D}'(\R^{d})$ and $\varphi\in \mathcal{D}(\R^{d})$. 
For more details on the theory of distributions see \cite{H1}.

By $\mathcal{O}(\Omega)$ we denote the space of $\C$-valued holomorphic functions 
on an open set $\Omega\subset\C$ and for $\alpha=(\alpha_{1},\alpha_{2})\in\N^{2}_{0}$ we often use the relation 
\begin{equation}\label{lem1}
\partial^{\alpha}f(z)=i^{\alpha_{2}}f^{(|\alpha|)}(z),\quad z\in\Omega,
\end{equation}
between real partial derivatives $\partial^{\alpha}f$ and complex derivatives $f^{(|\alpha|)}$ 
of a function $f\in\mathcal{O}(\Omega)$ (see e.g.\ \cite[3.4 Lemma, p.\ 17]{ich}).

%% file: SUPtoL2.tex
For applying H\"ormander's solution of the weighted $\overline{\partial}$-problem (see \cite[Chap.\ 4]{H3})
it is appropriate to consider weighted $L^{2}$-(semi)norms
and use them to control the seminorms $|\cdot|_{n,m}$ of solutions $u_{n}$ of $\overline{\partial}u_{n}=f$ 
in weighted $L^{2}$-spaces on $\Omega_{n}$ for given $f\in\mathcal{EV}(\Omega)$.

Throughout this section let $P$ be a polynomial in $d$ real variables with constant coefficients in $\K$, 
i.e.\ there are $n\in\N_{0}$ and $c_{\alpha}\in\K$ for $\alpha=(\alpha_{i})\in\N_{0}^{d},$ $|\alpha|\leq n$, such that
\[
P(\zeta)=\sum_{\substack{\alpha\in\N^{d}_{0},\\|\alpha|\leq n}}c_{\alpha}\zeta^{\alpha},\quad \zeta=(\zeta_{i})\in\R^{d},
\]
where $\zeta^{\alpha}:=\zeta^{\alpha_{1}}_{1}\cdots\zeta^{\alpha_{d}}_{d}$, and 
$P(\partial)$ be the linear partial differential operator associated to $P$.

\begin{lem}\label{lem:iso_sobolev}
Let $U\subset \R^{d}$ be open, $\{K_{n}\;|\; n\in\N\}$ a compact exhaustion of $U$, $P(\partial)$ a hypoelliptic 
partial differential operator and $q\in\N$. Then
\[
\mathcal{I}\colon \mathcal{C}^{\infty}(U)\rightarrow 
\mathcal{F}(U):=\{f\in L^{q}_{loc}(U)\;|\;\forall\;\alpha\in\N^{d}_{0}:\;
\partial^{\alpha}P(\partial)f\in L^{q}_{loc}(U)\},\;\mathcal{I}(f):=[f],
\]
is a topological isomorphism where $[f]$ is the equivalence class of $f$, the first space 
is equipped with the system of seminorms $\{|\cdot|_{K_{n},m}\;|\;n\in\N,\,m\in\N_{0}\}$ defined by
\begin{equation}\label{lem8.1}
|f|_{K_{n},m}:=\sup_{\substack{x\in K_{n} \\ \alpha\in\N^{d}_{0},|\alpha|\leq m}}|\partial^{\alpha}f(x)|,
\quad f\in \mathcal{C}^{\infty}(U),
\end{equation}
and the latter with the system
\begin{equation}\label{lem8.2}
\{\|\cdot\|_{L^{q}(K_{n})}+s_{K_{n},m}\;|\;\;n\in\N,\,m\in\N_{0}\}
\end{equation}
defined for $f=[F]\in\mathcal{F}(U)$ by 
\[
\|f\|_{L^{q}(K_{n})}:=\|F\|_{L^{q}(K_{n})}:=\bigl(\int_{K_{n}}{|F(x)|^{q} \d x}\bigr)^{\frac{1}{q}}
\]
and 
\[
s_{K_{n},m}(f):=\sup_{\alpha\in\N^{d}_{0},|\alpha|\leq m}\|\partial^{\alpha}P(\partial)f\|_{L^{q}(K_{n})}.
\]

\end{lem}
\begin{proof}
First, let us remark the following. 
The derivatives in the definition of $\mathcal{F}(U)$ are considered in the distributional sense and 
$\partial^{\alpha}P(\partial)f\in L^{q}_{loc}(U)$ means that 
there exists $g\in L^{q}_{loc}(U)$ such that $\partial^{\alpha}P(\partial)T_{f}=T_{g}$. 
The definition of the seminorm $\|\cdot\|_{L^{q}(K_{n})}$ 
does not depend on the chosen representative and we make no strict difference between an element 
of $L^{q}_{loc}(U)$ and its representative.

$(i)$ $\mathcal{C}^{\infty}(U)$, equipped with the system of seminorms \eqref{lem8.1}, is known to be a Fr\'echet space.
The space $\mathcal{F}(U)$, equipped with the system of seminorms \eqref{lem8.2}, is a metrisable locally convex space. 
Let $(f_{k})_{k\in\N}$ be a Cauchy sequence in 
$\mathcal{F}(U)$. By definition of $\mathcal{F}(U)$ we get for all $\beta\in\N^{d}_{0}$ that there exists a sequence
$(g_{k,\beta})_{k\in\N}$ in $L^{q}_{loc}(U)$ such that $\partial^{\beta}P(\partial)T_{f_{k}}=T_{g_{k,\beta}}$. 
Therefore we conclude from \eqref{lem8.2} that $(f_{k})_{k\in\N}$ and $(g_{k,\beta})_{k\in\N}$, 
$\beta\in\N^{d}_{0}$, are Cauchy sequences in $(L^{q}_{loc}(U),(\|\cdot\|_{L^{q}(K_{n})})_{n\in\N})$, 
which is a Fr\'echet space by \cite[5.17 Lemma, p.\ 36]{F/W/Buch}, so they have a limit $f$ resp.\ $g_{\beta}$ in this space. 
Since $(f_{k})_{k\in\N}$ resp.\ $(g_{k,\beta})_{k\in\N}$ converges to $f$ resp.\ $g_{\beta}$ in $L^{q}_{loc}(U)$, 
it follows that $(T_{f_{k}})_{k\in\N}$ resp.\ $(T_{g_{k,\beta}})_{k\in\N}$ converges to 
$T_{f}$ resp.\ $T_{g_{\beta}}$ in $\mathcal{D}_{\sigma}'(U)$. Here $\mathcal{D}_{\sigma}'(U)$ is the space 
$\mathcal{D}'(U)$ equipped with the weak$^{\ast}$-topology. Hence we get
\[	
\partial^{\beta}P(\partial)T_{f}\underset{k\to \infty}{\leftarrow}
\partial^{\beta}P(\partial)T_{f_{k}}=T_{g_{k,\beta}}\underset{k\to \infty}{\rightarrow}T_{g_{\beta}}
\]	
in $\mathcal{D}_{\sigma}'(U)$, implying $f\in \mathcal{F}(U)$ and the convergence of $(f_{k})_{k\in\N}$ 
to $f$ in $\mathcal{F}(U)$ with respect to the seminorms \eqref{lem8.2} as well. 
Thus this space is complete and so a Fr\'echet space.

$(ii)$ $\mathcal{I}$ is obviously linear and injective. It is continuous as for all $n\in\N$ and $m\in\N_{0}$ we have
\[
\|\mathcal{I}(f)\|^{q}_{L^{q}(K_{n})}\leq \lambda(K_{n})|f|_{K_{n},0}^{q}
\]
and there exists $C>0$, only depending on the coefficients and the number of summands of $P(\partial)$, such that
\[
s_{n,m}(\mathcal{I}(f))^{q}\leq C^{q}\lambda(K_{n})|f|_{K_{n},\operatorname{deg}P+m}^{q}
\]
for all $f\in \mathcal{C}^{\infty}(U)$ where $\lambda$ denotes the Lebesgue measure.

$(iii)$ The next step is to prove that $\mathcal{I}$ is surjective. Let $f\in \mathcal{F}(U)$. 
Then we have $P(\partial)f\in W^{\infty,q}_{loc}(U)$ where
\[
W^{\infty,q}_{loc}(U):=\{f\in L^{q}_{loc}(U)\;|\;\forall\;\alpha\in\N^{d}_{0}:
\;\partial^{\alpha}f\in L^{q}_{loc}(U)\}
\]
and so $P(\partial)f\in\mathcal{C}^{\infty}(U)$ by the Sobolev embedding theorem 
\cite[Theorem 4.5.13, p.\ 123]{H1} in combination with \cite[Theorem 3.1.7, p.\ 59]{H1}. 
To be precise, this means that the regular distribution 
$P(\partial)f$ has a representative in $\mathcal{C}^{\infty}(U)$. 
Due to the hypoellipticity of $P(\partial)$ we obtain $f\in\mathcal{C}^{\infty}(U)$, more precisely, 
that $f$ has a representative in $\mathcal{C}^{\infty}(U)$, so $\mathcal{I}$ is surjective. 

Finally, our statement follows from $(i)-(iii)$ and the open mapping theorem.
\end{proof}

\begin{cor}\label{cor:iso_sobolev} 
Let $P(\partial)$ be a hypoelliptic partial differential operator, $q\in\N$ and $0<r_{0}<r_{1}<r_{2}$. Then we have
\begin{flalign*}
\forall\; m\in\N_{0}\;&\exists\; l\in\N_{0},\, C>0\;\forall \;\alpha\in\N^{d}_{0},\,|\alpha|\leq m, \,
f\in\mathcal{C}^{\infty}(\mathring{Q}_{r_{2}}(0)): \\
&\|\partial^{\alpha}f\|_{Q_{r_{0}}(0)}\leq
C\bigl(\|f\|_{L^{q}(Q_{r_{1}}(0))}+\sup_{\beta\in\N^{d}_{0},|\beta|\leq l}
\|\partial^{\beta}P(\partial)f\|_{L^{q}(Q_{r_{1}}(0))}\bigr)
\end{flalign*}
where $Q_{r}(0):=[-r,r]^{d}$, $r>0$.
\end{cor}
\begin{proof}
Let $U:=\mathring{Q}_{r_{1}}(0)$. Then the sets $K_{n}:=Q_{r_{1}-\frac{1}{n+\nicefrac{1}{r_{1}}}}(0)$, $n\in\N$, 
form a compact exhaustion of $U$ and there exists $n_{0}=n_{0}(r_{0},r_{1})\in\N$ 
such that $Q_{r_{0}}(0)\subset K_{n_{0}}$.
Since $\mathcal{I}^{-1}\colon \mathcal{F}(U)\rightarrow\mathcal{C}^{\infty}(U)$ is continuous 
by \prettyref{lem:iso_sobolev}, there are $N\in\N$, $l\in\N_{0}$ and $C>0$ such that
\begin{align*}
\|\partial^{\alpha}f\|_{Q_{r_{0}}(0)}&\leq |f|_{K_{n_{0}},m} 
=|\mathcal{I}^{-1}([f])|_{K_{n_{0}},m}
\leq C\bigl(\|[f]\|_{L^{q}(K_{N})}+s_{K_{N},l}([f])\bigr)\\
&\leq C\bigl(\|f\|_{L^{q}(Q_{r_{1}}(0))}+\sup_{\beta\in\N^{d}_{0},|\beta|\leq l}
\|\partial^{\beta}P(\partial)f\|_{L^{q}(Q_{r_{1}}(0))}\bigr)
\end{align*}
for all $\alpha\in\N^{d}_{0}$, $|\alpha|\leq m$, and $f\in\mathcal{C}^{\infty}(\mathring{Q}_{r_{2}}(0))$.
\end{proof}

Due to this corollary we can switch to types of $L^{q}$-seminorms which induce the same topology on $\mathcal{EV}(\Omega)$ 
as the $\sup$-seminorms and we get an useful inequality to control the growth of the solutions of the weighted 
$\overline{\partial}$-problem by the right-hand side under the following conditions.

\begin{condPN}\label{cond:weights}
Let $\mathcal{V}:=(\nu_{n})_{n\in\N}$ be a directed family of continuous weights on an open set $\Omega\subset\R^{2}$ 
and $(\Omega_{n})_{n\in\N}$ a family of non-empty open sets such that $\Omega_{n}\subset\Omega_{n+1}$
and $\Omega=\bigcup_{n\in\N} \Omega_{n}$.
For every $k\in\N$ let there be $\rho_{k}\in \R$ such that $0<\rho_{k}<\d^{\|\cdot\|_{\infty}}(\{x\},\partial\Omega_{k+1})$
for all $x\in\Omega_{k}$ and let there be $q\in\N$ such that for any $n\in\N$ there are
$\psi_{n}\in L^{q}(\Omega_{k})$, $\psi_{n}>0$, and $\N\ni J_{i}(n)\geq n$ and $C_{i}(n)>0$ for $i=1,2$ 
such that for any $x\in\Omega_{k}$:
\begin{itemize}
  \item [$(PN.1)\phantom{^{q}}$] $\sup_{\zeta\in\R^{2},\,\|\zeta\|_{\infty}\leq \rho_{k}}\nu_{n}(x+\zeta)
  \leq C_{1}(n)\inf_{\zeta\in\R^{2},\,\|\zeta\|_{\infty}\leq \rho_{k}}\nu_{J_{1}(n)}(x+\zeta)$
  \item [$(PN.2)^{q}$] $\nu_{n}(x)\leq C_{2}(n)\psi_{n}(x)\nu_{J_{2}(n)}(x)$
\end{itemize}
\end{condPN}

If $q=1$, then these conditions are a special case of \cite[Condition 2.1, p.\ 176]{kruse2018_4} 
by \cite[Remark 2.3 (b), p.\ 177]{kruse2018_4} and modifications of the conditions $(1.1)$-$(1.3)$ 
in \cite[p.\ 204]{L4}. 
They guarantee that the \textbf{p}rojective limit $\mathcal{EV}(\Omega)$ is a \textbf{n}uclear Fr\'echet space by \cite[Theorem 3.1, p.\ 188]{kruse2018_4} 
and \cite[Remark 2.7, p.\ 178-179]{kruse2018_4} if $q=1$, which we use to derive the surjectivity of 
$\overline{\partial}^{E}$ from the one of $\overline{\partial}$ for Fr\'echet spaces $E$ over $\C$.

\begin{rem}\label{rem:standard_psi}
A typical candidate for $\psi_{n}$ for every $n\in\N$ is $\psi_{n}(x):= (1+|x|^{2})^{-d}$, $x\in\R^{d}$. If
Condition $(PN)$ is fulfilled with this choice for $\psi_{n}$ for some $q\in\N$, then it is fulfilled 
for every $q\in\N$ because $\psi_{n}\in L^{q}(\R^{d})$ for every $q\in\N$.
\end{rem}

\begin{conv}[{\cite[1.1 Convention, p.\ 205]{L4}}]\label{conv:index}
We often delete the number $n$ counting the seminorms (e.g.\ $J_{i}=J_{i}(n)$ or $C_{i}=C_{i}(n)$) and 
indicate compositions with the functions $J_{i}$ only in the index 
(e.g.\ $J_{23}=J_{2}(J_{3}(n))$).
\end{conv}

\begin{defn}
Let $\Omega\subset\R^{d}$ be open and $(\Omega_{n})_{n\in\N}$ a family of non-empty
open sets such that $\Omega_{n}\subset\Omega_{n+1}$ and $\Omega=\bigcup_{n\in\N} \Omega_{n}$.
Let $\mathcal{V}:=(\nu_{n})_{n\in\N}$ be a (directed) family of continuous weights on $\Omega$.
For $n,q\in\N$ we define the locally convex Hausdorff spaces
\[
\mathcal{C}^{\infty}_{q}\nu_{n}(\Omega_{n}):=\{f\in\mathcal{C}^{\infty}(\Omega_{n})\;|\; 
\forall\; m\in\N_{0}:\;\|f\|_{n,m,q}<\infty\}
\]
and 
\[
\mathcal{C}^{\infty}_{q}\mathcal{V}(\Omega)
:=\{ f \in \mathcal{C}^{\infty}(\Omega)\; | \;\forall n \in \N:\; 
f_{\mid\Omega_{n}}\in \mathcal{C}^{\infty}_{q}\nu_{n}(\Omega_{n})\}
\]
where
\[
\|f\|_{n,m,q}:=\sup_{\alpha\in\N^{d}_{0},|\alpha|\leq m}
\bigl(\int_{\Omega_{n}}{|\partial^{\alpha}f(x)|^{q}\nu_{n}(x)^{q}\d x}\bigr)^{\frac{1}{q}}.
\]
\end{defn}

\begin{lem}\label{lem:switch_top}
Let $(PN)$ be fulfilled for some $q\in\N$.
\begin{enumerate}
\item [a)] Let $P(\partial)$ be a hypoelliptic partial differential operator, $n\in\N$
and $f\in\mathcal{C}^{\infty}(\Omega_{2J_{11}})$ such that $\|f\|_{2J_{11},0,q}<\infty$ and 
$P(\partial)f\in \mathcal{C}^{\infty}_{q}\nu_{2J_{11}}(\Omega_{2J_{11}})$. 
Then $f\in\mathcal{E}\nu_{n}(\Omega_{n})$ and 
\[
\forall\;m\in\N_{0}\;\exists\; l\in\N_{0},\, C_{0}>0:\;
|f|_{n,m}\leq C_{0}(\|f\|_{2J_{11},0,q}+\|P(\partial)f\|_{2J_{11},l,q}).
\]
\item [b)] Then $\mathcal{C}^{\infty}_{q}\mathcal{V}(\Omega)=\mathcal{EV}(\Omega)$ as locally convex spaces.
\end{enumerate}
\end{lem}
\begin{proof}
$a)$ Due to \cite[Lemma 2.11 (p.1), p.\ 183-184]{kruse2018_4}, \cite[Remark 2.7, p.\ 178-179]{kruse2018_4} 
and \cite[Remark 2.3 (b), p.\ 177]{kruse2018_4} there are $\mathcal{K}\subset\N$ and a sequence 
$(z_{k})_{k\in\mathcal{K}}$, $z_{k}\neq z_{j}$ for $k\neq j$, in $\Omega_{n}$ such that the balls
 \[
  b_{k}:=\{\zeta\in\R^{d}\;|\;\|\zeta-z_{k}\|_{\infty}<\rho_{n}/2\}
 \]
form an open covering of $\Omega_{n}$ with
\[
 \Omega_{n}\subset \bigcup_{k\in\mathcal{K}}b_{k}\subset\bigcup_{k\in\mathcal{K}}B_{k}
 \subset \Omega_{n+1}\subset \Omega_{2J_{11}(n)}
\]
where 
\[
 B_{k}:=\{\zeta\in\R^{d}\;|\;\|\zeta-z_{k}\|_{\infty}<\rho_{n}\}.
\]
Let $m\in\N_{0}$, $\alpha\in\N^{d}_{0}$, $|\alpha|\leq m$, and $k\in\mathcal{K}$. By \prettyref{cor:iso_sobolev} there exist 
$l\in\N_{0}$ and $C>0$, $C$ and $l$ independent of $k$ and $\alpha$, such that
\begin{flalign}\label{lem10.1}
&\hspace{0.37cm} \|(\partial^{\alpha}f)\nu_{n}\|_{b_{k}}\nonumber\\
&\leq \|\nu_{n}\|_{b_{k}}\|\partial^{\alpha}f\|_{b_{k}}
\underset{(\omega.1)}{\leq}C_{1}\nu_{J_{1}}(z_{k})\|\partial^{\alpha}f\|_{b_{k}}\nonumber\\
&= C_{1}\nu_{J_{1}}(z_{k})\|\partial^{\alpha}f(z_{k}+\cdot)\|_{Q_{\rho_{n}/2}(0)}\nonumber\\
&\underset{\mathclap{\ref{cor:iso_sobolev}}}{\leq} CC_{1}\nu_{J_{1}}(z_{k})
 \bigl(\|f\|_{L^{q}(B_{k})} 
 +\sup_{\beta\in\N^{2}_{0},|\beta|\leq l}
 \|\partial^{\beta}P(\partial)f\|_{L^{q}(B_{k})}\bigr)\nonumber\\
&\underset{\mathclap{(PN.1)}}{\leq} CC_{1}C_{1}(J_{1})
 \bigl(\|f\nu_{J_{11}}\|_{L^{q}(B_{k})} 
 +\sup_{\beta\in\N^{2}_{0},|\beta|\leq l}
 \|(\partial^{\beta}P(\partial)f)\nu_{J_{11}}\|_{L^{q}(B_{k})}\bigr)\nonumber\\
&\leq CC_{1}C_{1}(J_{1})\bigl(\|f\|_{2J_{11},0,q}+\|P(\partial)f\|_{2J_{11},l,q}\bigr)
\end{flalign}
and so we get
\begin{align*}
|f|_{n,m}&\leq \sup_{\substack{x\in\bigcup_{k\in \mathcal{K}}b_{k}\\ \alpha\in\N^{d}_{0},\,|\alpha|\leq m}}
{|\partial^{\alpha}f(x)|\nu_{n}(x)}
\leq\quad\sup_{\mathclap{\substack{k\in \mathcal{K}\\ \alpha\in\N^{d}_{0},\,|\alpha|\leq m}}}
{\|(\partial^{\alpha}f)\nu_{n}\|_{b_{k}}}\\
&\underset{\mathclap{\eqref{lem10.1}}}{\leq}
CC_{1}C_{1}(J_{1})\bigl(\|f\|_{2J_{11},0,q}+\|P(\partial)f\|_{2J_{11},l,q}\bigr).
\end{align*}

$b)$ Let $f\in\mathcal{C}^{\infty}_{q}\mathcal{V}(\Omega)$ and $P(\partial):=\Delta$ be the Laplacian. 
Then $f$ satisfies the conditions of $a)$ for all $n\in\N$ because
\[
\|\Delta f\|_{n,m,q}=\sup_{\alpha\in\N^{d}_{0},|\alpha|\leq m}
\|\sum_{i=1}^{d}(\partial^{\alpha+2e_{i}}f)\nu_{n}\|_{L^{q}(\Omega_{n})}
\leq d \|f\|_{n,m+2,q}<\infty
\]
for every $m\in\N_{0}$. So for every $n\in\N$ and $m\in\N_{0}$ there exist $l\in\N_{0}$ and $C_{0}>0$ such that
\begin{align*}
|f|_{n,m}&\leq C_{0}\bigl(\|f\|_{2J_{11},0,q}+\|\Delta f\|_{2J_{11},l,q}\bigr)\\
&\leq C_{0}\bigl(\|f\|_{2J_{11},0,q}+d\|f\|_{2J_{11},l+2,q}\bigr)
\leq (1+d)C_{0}\|f\|_{2J_{11},l+2,q}.
\end{align*}
On the other hand, let $f\in\mathcal{EV}(\Omega)$. For every $n\in\N$ and $m\in\N_{0}$ we have
\begin{align*}
\|f\|_{n,m,q}^{q}&=\sup_{\alpha\in\N^{d}_{0},|\alpha|\leq m}
\int_{\Omega_{n}}{|\partial^{\alpha}f(x)|^{q}\nu_{n}(x)^{q}\d x}\\
&\underset{\mathclap{(PN.2)^{q}}}{\leq}\quad C_{2}^{q}\sup_{\alpha\in\N^{d}_{0},|\alpha|\leq m}
\int_{\Omega_{n}}{|\partial^{\alpha}f(x)|^{q}\psi_{n}(x)^{q}\nu_{J_{2}}(x)^{q}\d x}\\
&\leq C_{2}^{q}\int_{\Omega_{n}}{\psi_{n}(x)^{q}\d x}\sup_{\substack{w\in \Omega_{n}\\ \alpha\in\N^{d}_{0},\,|\alpha|\leq m}}
|\partial^{\alpha}f(w)|^{q}\nu_{J_{2}}(w)^{q}\\
&\leq C_{2}^{q}\|\psi\|^{q}_{L^{q}(\Omega_{n})}\left|f\right|_{J_{2},m}^{q}.
\end{align*}
\end{proof}

The following examples from \cite[Example 2.8, p.\ 179]{kruse2018_4} and \cite[Example 2.9, p.\ 182]{kruse2018_4} 
fulfil $(PN)$ for every $q\in\N$ (see \prettyref{rem:standard_psi}).

\begin{exa}\label{ex:families_of_weights_1}
Let $\Omega\subset\R^{d}$ be open and $(\Omega_{n})_{n\in\N}$ a family of non-empty open sets such that
\begin{enumerate}
 \item [(i)] $\Omega_{n}:=\R^{d}$ for every $n\in\N$.
 \item [(ii)] $\Omega_{n}\subset\Omega_{n+1}$ and $\d^{\|\cdot\|}(\Omega_{n},\partial\Omega_{n+1})>0$ for every $n\in\N$.
 \item [(iii)] $\Omega_{n}:=\{x=(x_{i})\in \Omega\;|\;\forall\;i\in I:\;|x_{i}|<n+N\;\text{and}\;
 \d^{\|\cdot\|}(\{x\},\partial \Omega)>1/(n+N) \}$ where $I\subset\{1,\ldots,d\}$, $\partial \Omega\neq\varnothing$ and $N\in\N_{0}$ 
 is big enough.
 \item [(iv)] $\Omega_{n}:=\{x=(x_{i})\in \Omega\;|\;\forall\;i\in I:\;|x_{i}|<n\}$ where $I\subset\{1,\ldots,d\}$ and $\Omega:=\R^{d}$.
 \item [(v)] $\Omega_{n}:=\mathring K_{n}$ where $K_{n}\subset\mathring K_{n+1}$, $\mathring{K}_{n}\neq\varnothing$, 
 is a compact exhaustion of $\Omega$.
\end{enumerate}
Let $(a_{n})_{n\in\N}$ be strictly increasing such that $a_{n}\geq 0$ for all $n\in\N$ or 
$a_{n}\leq 0$ for all $n\in\N$. The family $\mathcal{V}:=(\nu_{n})_{n\in\N}$ of positive continuous functions on $\Omega$ given by 
\[
 \nu_{n}\colon\Omega\to (0,\infty),\;\nu_{n}(x):=e^{a_{n}\mu(x)},
\]
with some function $\mu\colon\Omega\to[0,\infty)$ fulfils $\nu_{n}\leq\nu_{n+1}$ for all $n\in\N$ and
$(PN)$ for every $q\in\N$ with $\psi_{n}(x):= (1+|x|^{2})^{-d}$, $x\in\R^{d}$, for every $n\in\N$ if
\begin{enumerate}
  \item [a)] there is some $0<\gamma\leq 1$ and such that
  $\mu(x)=|(x_{i})_{i\in I_{0}}|^{\gamma}$, $x=(x_{i})\in\Omega$, where $I_{0}:=\{1,\ldots,d\}\setminus I$ with $I\subsetneq\{1,\ldots,d\}$ and 
  $(\Omega_{n})_{n\in\N}$ from (iii) or (iv).
  \item [b)] $\lim_{n\to\infty}a_{n}=\infty$ or $\lim_{n\to\infty}a_{n}=0$ and there is some 
  $m\in\N$, $m\leq 2d+1$, such that $\mu(x)=|x|^{m}$, $x\in\Omega$, with $(\Omega_{n})_{n\in\N}$ from (i) or (ii).
  \item [c)] $a_{n}=n/2$ for all $n\in\N$ and $\mu(x)=\ln(1+|x|^{2})$, $x\in\R^{d}$, with $(\Omega_{n})_{n\in\N}$ from (i).
  \item [d)] $\mu(x)=0$, $x\in\Omega$, with $(\Omega_{n})_{n\in\N}$ from (v).
\end{enumerate}
\end{exa}

\prettyref{ex:families_of_weights_1} a) covers the weights $\nu_{n}(z):=\exp{(-|\re(z)|/n)}$, $z\in \C\setminus\R$, 
with the sets $\Omega_{n}:=\{z\in\C\;|\; 1/n<|\im(z)|<n\}$ and \prettyref{lem:switch_top} is a 
generalisation of \cite[5.15 Lemma, p.\ 78]{ich}. 
In \prettyref{ex:families_of_weights_1} c) we have $\mathcal{EV}(\R^{d})=\mathcal{S}(\R^{d})$, the Schwartz space. 
Hence the $\sup$- and $L^{q}$-seminorms form an equivalent system of seminorms on $\mathcal{S}(\R^{d})$ for every $q\in\N$ 
by \prettyref{lem:switch_top}. This generalises the observation for $q=2$ and $d=1$ in \cite[Example 29.5 (2), p.\ 361]{meisevogt1997} 
(cf.\ \cite[Folgerung, p.\ 85]{Wloka1967}).
In \prettyref{ex:families_of_weights_1} d) $\mathcal{EV}(\Omega)=\mathcal{C}^{\infty}(\Omega)$ with the usual topology 
of uniform convergence of partial derivatives of any order on compact subsets of $\Omega$.
We remark that we can choose $C_{2}(n):=(1+\sup\{|x|^2\;|\;x\in\Omega_{n}\})^{d}$ in d). 
Other choices for $\psi_{n}$ are also possible, for example, $\psi_{n}:=1$ in d) 
but we are interested in this particular choice of $\psi_{n}$ in view of the next section.

%% file: main_thm.tex
In this section we always assume that $\K=\C$ and $d=2$ and state our main theorem on the surjectivity of the Cauchy-Riemann operator
\[
\overline{\partial}^{E}:=\frac{1}{2}((\partial_{1})^{E}+i(\partial_{2})^{E})
\colon \mathcal{EV}(\Omega,E)\to \mathcal{EV}(\Omega,E)
\]
for any Fr\'echet space $E$. Let us outline the proof. 
If the subspace of the projective spectrum $\mathcal{EV}(\Omega)$ consisting of holomorphic functions 
has some kind of density property weaker than weak reducibility 
and for every $f\in\mathcal{EV}(\Omega)$ and $n\in\N$ there is $u_{n}\in\mathcal{E}\nu_{n}(\Omega_{n})$ such that 
$\overline{\partial} u_{n}=f$ on $\Omega_{n}$,
then the Mittag-Leffler procedure yields the surjectivity of $\overline{\partial}^{E}$ for $E=\C$. 
Since the space $\mathcal{EV}(\Omega)$ is a nuclear Fr\'echet space under condition $(PN)$ and 
$\mathcal{EV}(\Omega,E)\cong\mathcal{EV}(\Omega)\widehat{\otimes}_{\pi}E$, 
the surjectivity of $\overline{\partial}$ implies the surjectivity of $\overline{\partial}^{E}$ 
for any Fr\'echet space $E$ by the classical theory of tensor products of Fr\'echet spaces.

\begin{defn}
Let $E$ be a locally convex Hausdorff space over $\C$, 
$\mathcal{V}:=(\nu_{n})_{n\in\N}$ a (directed) family of continuous weights on an open set
$\Omega\subset\R^{2}$ and $(\Omega_{n})_{n\in\N}$ a family of non-empty
open sets such that $\Omega_{n}\subset\Omega_{n+1}$ for all $n\in\N$ and $\Omega=\bigcup_{n\in\N} \Omega_{n}$.
For $n\in\N$ we define
\[
\mathcal{E}\nu_{n,\overline{\partial}}(\Omega_{n},E):= \{ f \in \mathcal{E}\nu_{n}(\Omega_{n}, E)\; | \;
f\in\operatorname{ker}\overline{\partial}^{E} \}
\]
and 
\[
\mathcal{EV}_{\overline{\partial}}(\Omega, E):=\{f\in\mathcal{EV}(\Omega, E)\;|\;
f\in\operatorname{ker}\overline{\partial}^{E}\}.
\]
$\mathcal{EV}_{\overline{\partial}}(\Omega)$ is called \emph{very weakly reduced} if for every $n\in\N$ there 
are $\iota_{1}(n),\iota_{2}(n)\in\N$, $\iota_{2}(n)\geq \iota_{1}(n)$, such that the restricted space $\pi_{\iota_{2}(n),n}(\mathcal{E}\nu_{\iota_{2}(n),\overline{\partial}}(\Omega_{\iota_{2}(n)}))$ is dense in $\pi_{\iota_{1}(n),n}(\mathcal{E}\nu_{\iota_{1}(n),\overline{\partial}}(\Omega_{\iota_{1}(n)}))$ w.r.t.\ $(|\cdot|_{n,m})_{m\in\N_{0}}$.
\end{defn}

We note that $\mathcal{EV}_{\overline{\partial}}(\Omega)$ is very weakly reduced if it is weakly reduced. 
By now all ingredients that are required to prove the surjectivity of $\overline{\partial}$ are provided.

\begin{thm}\label{thm:scalar_CR_surjective}
Let $(PN)$ with $\psi_{n}(z):=(1+|z|^{2})^{-2}$, $z\in\Omega$, be fulfilled for some (thus all) $q\in\N$, 
$\mathcal{EV}_{\overline{\partial}}(\Omega)$ be very weakly reduced with $\iota_{2}(n)\geq \iota_{1}(n+1)$
and $-\ln\nu_{n}$ be subharmonic on $\Omega$ for every $n\in\N$. Then
\[
\overline{\partial}\colon \mathcal{EV}(\Omega)\to\mathcal{EV}(\Omega)
\]
is surjective.
\end{thm}
\begin{proof}
$(i)$ Let $f\in\mathcal{EV}(\Omega)$, $n\in\N$ and set
	\[
	\varphi_{n}\colon\Omega\to\R,\;\varphi_{n}\left(z\right):=-2\ln\nu_{J_{2}(2J_{11}(n))}(z),
	\]
which is a (pluri)subharmonic function on $\Omega$. The set $\Omega_{J_{2}(2J_{11})}$ is open and 
pseudoconvex since every open set in $\C$ is a domain of holomorphy by 
\cite[Corollary 1.5.3, p.\ 15]{H3} and hence pseudoconvex by \cite[Theorem 4.2.8, p.\ 88]{H3}. 
For the differential form $g:=f\operatorname{d\bar{z}}$ we have $\overline{\partial}g=0$ 
in the sense of differential forms and $f\in\mathcal{EV}(\Omega)=\mathcal{C}^{\infty}_{2}\mathcal{V}(\Omega)$ 
by \prettyref{lem:switch_top} b) resulting in
\[
\int_{\Omega_{J_{2}(2J_{11})}}{|f(z)|^{2}e^{-\varphi_{n}(z)} \d z}= \|f\|_{J_{2}(2J_{11}),0,2}^{2}<\infty.
\]
Thus by \cite[Theorem 4.4.2, p.\ 94]{H3} there is a solution $u_{n}\in L^{2}_{loc}(\Omega_{J_{2}(2J_{11})})$ of
$\overline{\partial}u_{n}=f_{\mid\Omega_{J_{2}(2J_{11})}}$ in the distributional sense 
such that
\[
\int_{\Omega_{J_{2}(2J_{11})}}{|u_{n}(z)|^{2}e^{-\varphi_{n}(z)}(1+|z|^{2})^{-2} \d z}
\leq\int_{\Omega_{J_{2}(2J_{11})}}{|f(z)|^{2}e^{-\varphi_{n}(z)} \d z}.
\]
Since $\overline{\partial}$ is hypoelliptic, it follows that $u_{n}\in\mathcal{C}^{\infty}(\Omega_{J_{2}(2J_{11})})$, 
resp.\ $u_{n}$ has a representative which is $\mathcal{C}^{\infty}$. 
By virtue of property $(PN.2)^{2}$ we gain
\begin{align*}
 \|u_{n}\|_{2J_{11},0,2}^{2}
&=\int_{\Omega_{2J_{11}}}{|u_{n}(z)|^{2}\nu_{2J_{11}}(z)^{2} \d z}\\
&\underset{\mathclap{(PN.2)^{2}}}{\leq} \quad C_{2}(2J_{11})^{2}\int_{\Omega_{J_{2}(2J_{11})}}
 {|u_{n}(z)|^{2}e^{-\varphi_{n}(z)}(1+|z|^{2})^{-4} \d z}\\
&\leq C_{2}(2J_{11})^{2}\int_{\Omega_{J_{2}(2J_{11})}}{|u_{n}(z)|^{2}e^{-\varphi_{n}(z)}(1+|z|^{2})^{-2} \d z} <\infty .
\end{align*}
So the conditions of \prettyref{lem:switch_top} a) are fulfilled for all $n\in\N$, 
implying $u_{n}\in\mathcal{E}\nu_{n}(\Omega_{n})$.

$(ii)$ The next step is to prove the surjectivity of $\overline{\partial}\colon \mathcal{EV}(\Omega)\to\mathcal{EV}(\Omega)$ 
via the Mittag-Leffler procedure (see \cite[9.14 Theorem, p.\ 206-207]{Kaballo}). 
Due to $(i)$ we have for every $l\in\N$ a function $u_{l}\in\mathcal{E}\nu_{l}(\Omega_{l})$ 
such that $\overline{\partial}u_{l}=f_{\mid\Omega_{l}}$.
Now, we inductively construct $g_{n}\in\mathcal{E}\nu_{\iota_{1}(n)}(\Omega_{\iota_{1}(n)})$, $n\in\N$, such that
\begin{enumerate}
	\item [(1)]
	$\overline{\partial}g_{n}=f_{\mid\Omega_{\iota_{1}(n)}}$, $n\geq 1$,
	\item [(2)]
	$|g_{n}-g_{n-1}|_{n-1,n-1}\leq \frac{1}{2^{n}}$, $n\geq 2$,
\end{enumerate}
with $\iota_{1}(n)$ from the definition of very weak reducibility.
For $n=1$ set $g_{1}:=u_{\iota_{1}(1)}$. Then we have $g_{1}\in\mathcal{E}\nu_{\iota_{1}(1)}(\Omega_{\iota_{1}(1)})$ 
and $\overline{\partial}g_{1}=f_{\mid\Omega_{\iota_{1}(1)}}$ by part $(i)$. 
Let $g_{n}$ fulfil (1) for some $n\geq 1$. Since
\[
 \overline{\partial}(u_{\iota_{2}(n)}-g_{n})_{\mid\Omega_{\iota_{1}(n)}}
=\overline{\partial}{u_{\iota_{2}(n)}}_{\mid\Omega_{\iota_{1}(n)}}-{\overline{\partial}g_{n}}_{\mid\Omega_{\iota_{1}(n)}}
 \underset{(i),\,(1)}{=}f_{\mid\Omega_{\iota_{1}(n)}}-f_{\mid\Omega_{\iota_{1}(n)}}=0,
\]
it follows $u_{\iota_{2}(n)}-g_{n}\in\mathcal{E}\nu_{\iota_{1}(n),\overline{\partial}}(\Omega_{\iota_{1}(n)})$ 
and by the very weak reducibility of $\mathcal{EV}_{\overline{\partial}}(\Omega)$ there is 
$h_{n+1}\in\mathcal{E}\nu_{\iota_{2}(n),\overline{\partial}}(\Omega_{\iota_{2}(n)})$ such that
\[
 |u_{\iota_{2}(n)}-g_{n}-h_{n+1}|_{n,n}\leq \frac{1}{2^{n+1}}.
\]
Set $g_{n+1}:=u_{\iota_{2}(n)}-h_{n+1}\in\mathcal{E}\nu_{\iota_{2}(n)}(\Omega_{\iota_{2}(n)})$. 
As $\iota_{2}(n)\geq \iota_{1}(n+1)$, we have $g_{n+1}\in\mathcal{E}\nu_{\iota_{1}(n+1)}(\Omega_{\iota_{1}(n+1)})$.
Condition (2) is satisfied by the inequality above and condition (1) as well because
\[
 \overline{\partial}g_{n+1}
=\overline{\partial}u_{\iota_{2}(n)}-\overline{\partial}h_{n+1}
=\overline{\partial}u_{\iota_{2}(n)}-0
\underset{(i)}{=}f_{\mid\Omega_{\iota_{1}(n+1)}}.
\]
Now, let $\varepsilon>0$, $l\in\N$ and $m\in\N_{0}$. Choose $l_{0}\in\N$, $l_{0}\geq \max\left(l,m\right)$, 
such that $\frac{1}{2^{l_{0}}}<\varepsilon$. For all $p\geq k\geq l_{0}$ we get
\begin{align*}
 |g_{p}-g_{k}|_{l,m}
&\leq|g_{p}-g_{k}|_{l_{0},l_{0}}
 =\bigl|\sum^{p}_{j=k+1}{g_{j}-g_{j-1}}\bigr|_{l_{0},l_{0}}
 \leq\sum^{p}_{j=k+1}{|g_{j}-g_{j-1}|_{l_{0},l_{0}}}\\
&\underset{\mathclap{l_{0}\leq k\leq j-1}}{\leq}\quad\; \sum^{p}_{j=k+1}{|g_{j}-g_{j-1}|_{j-1,j-1}}
 \;\underset{\mathclap{(2)}}{\leq} \;\sum^{p}_{j=k+1}{\frac{1}{2^{j}}}<\frac{1}{2^{k}}
 \leq\frac{1}{2^{l_{0}}}<\varepsilon.
\end{align*}
Hence $(g_{n})_{n\geq \max(l-2,1)}$ is a Cauchy sequence in $\mathcal{E}\nu_{l}(\Omega_{l})$ for all $l\in\N$ and, 
since these spaces are complete by \cite[Proposition 3.7, p.\ 240]{kruse2018_2}, 
it has a limit $G_{l}\in\mathcal{E}\nu_{l}(\Omega_{l})$. These limits coincide 
on their common domain because for every $l_{1},l_{2}\in\N$, $l_{1}<l_{2}$, and $\varepsilon_{1}>0$ 
there exists $N\in\N$ such that for all $n\geq N$
\begin{align*}
 |G_{l_{1}}-G_{l_{2}}|_{l_{1},m}
&\leq |G_{l_{1}}-g_{n}|_{l_{1},m}+|g_{n}-G_{l_{2}}|_{l_{1},m}
 \leq |G_{l_{1}}-g_{n}|_{l_{1},m}+|g_{n}-G_{l_{2}}|_{l_{2},m}\\
&<\frac{\varepsilon_{1}}{2}+\frac{\varepsilon_{1}}{2}=\varepsilon_{1}.
\end{align*}
So the limit function $g$, defined by $g:=G_{l}$ on $\Omega_{l}$ for all $l\in\N$, is well-defined 
and we have $g\in\mathcal{EV}(\Omega)$.
Thus we obtain for all $l\in\N$
\[
 f_{\mid\Omega_{l}}
\underset{\substack{(1)\\n\geq \max(l-2,1)}}{=}{\overline{\partial}g_{n}}_{\mid\Omega_{l}}
\underset{n\to\infty}{\to} \overline{\partial}g_{\mid\Omega_{l}}
\]
and hence the existence of $g\in\mathcal{EV}(\Omega)$ with $\overline{\partial}g=f$ on $\Omega$ is proved.
\end{proof}

Moreover, we are already able to show that $\overline{\partial}^{E}$ is surjective for Fr\'echet spaces $E$ over $\C$
just by using classical theory of tensor products of Fr\'echet spaces.

\begin{cor}\label{cor:frechet_CR_surjective}
Let $(PN)$ with $\psi_{n}(z):=(1+|z|^{2})^{-2}$, $z\in\Omega$, be fulfilled for some (thus all) $q\in\N$, 
$\mathcal{EV}_{\overline{\partial}}(\Omega)$ be very weakly reduced with $\iota_{2}(n)\geq \iota_{1}(n+1)$
and $-\ln\nu_{n}$ be subharmonic on $\Omega$ for every $n\in\N$.
If $E$ is a Fr\'echet space over $\C$, then
\[
\overline{\partial}^{E}\colon \mathcal{EV}(\Omega,E)\to\mathcal{EV}(\Omega,E)
\]
is surjective.
\end{cor}
\begin{proof}
First, we recall some definitions and facts from the theory of tensor products 
(see \cite{Defant}, \cite{Jarchow}, \cite{Kaballo}). The $\varepsilon$-product of Schwartz 
is given by $\mathcal{EV}(\Omega)\varepsilon E:=L_{e}(\mathcal{EV}(\Omega)_{\kappa}',E)$ where 
the dual $\mathcal{EV}(\Omega)'$ is equipped with the topology of uniform convergence on 
absolutely convex, compact subsets of $\mathcal{EV}(\Omega)$ and $L(\mathcal{EV}(\Omega)_{\kappa}',E)$ 
with the topology of uniform convergence on equicontinuous subsets of $\mathcal{EV}(\Omega)'$. 
The space $\mathcal{EV}(\Omega)$ is a Fr\'echet space by \cite[Proposition 3.7 , p.\ 240]{kruse2018_2} 
and nuclear by \cite[Theorem 3.1, p.\ 188]{kruse2018_4}, \cite[Remark 2.7, p.\ 178-179]{kruse2018_4} 
and \cite[Remark 2.3 (b), p.\ 177]{kruse2018_4} because $(PN.1)$ and $(PN.2)^{1}$ are fulfilled.
Thus the map 
\[
S\colon \mathcal{EV}(\Omega)\varepsilon E\to\mathcal{EV}(\Omega,E),\; 
u\longmapsto [x\mapsto u(\delta_{x})],
\]
is a topological isomorphism by \cite[Example 16 c), p.\ 1526]{kruse2017} 
where $\delta_{x}$ is the point-evaluation at $x\in\Omega$. 
The continuous linear injection
\[
\chi\colon \mathcal{EV}(\Omega)\otimes_{\pi} E \to \mathcal{EV}(\Omega)\varepsilon E,\; 
\sum_{n=1}^{k} f_{n}\otimes e_{n} \longmapsto \bigl[y\mapsto \sum_{n=1}^{k} y(f_{n})e_{n}\bigr],
\]
from the tensor product $\mathcal{EV}(\Omega)\otimes_{\pi} E$ with the projective topology 
extends to a continuous linear map 
$\widehat{\chi}\colon \mathcal{EV}(\Omega)\widehat{\otimes}_{\pi} E \to \mathcal{EV}(\Omega)\varepsilon E$
on the completion $\mathcal{EV}(\Omega)\widehat{\otimes}_{\pi} E$ of $\mathcal{EV}(\Omega)\otimes_{\pi} E$. 
The map $\widehat{\chi}$ is also a topological isomorphism since $\mathcal{EV}(\Omega)$ is nuclear.
Furthermore, we define
\[
\overline{\partial}\varepsilon \operatorname{id}_{E}
\colon \mathcal{EV}(\Omega)\varepsilon E \to \mathcal{EV}(\Omega)\varepsilon E,\;
u\mapsto u\circ \overline{\partial}^{t},
\]
where $\overline{\partial}^{t}\colon \mathcal{EV}(\Omega)'\to\mathcal{EV}(\Omega)'$, $y\mapsto y\circ\overline{\partial}$, 
and $\overline{\partial} \otimes_{\pi}\operatorname{id}_{E}\colon 
\mathcal{EV}(\Omega)\otimes_{\pi}E\to \mathcal{EV}(\Omega)\otimes_{\pi}E$ 
is defined by the relation $\chi\circ (\overline{\partial} \otimes_{\pi}\operatorname{id}_{E})
=(\overline{\partial}\varepsilon \operatorname{id}_{E})\circ \chi $. Denoting by 
$\overline{\partial}\, \widehat{\otimes}_{\pi}\operatorname{id}_{E}$ the continuous linear extension of 
$\overline{\partial} \otimes_{\pi}\operatorname{id}_{E}$ to the completion $\mathcal{EV}(\Omega)\widehat{\otimes}_{\pi} E$, 
we observe that 
\[
 \overline{\partial}^{E}
=S\circ(\overline{\partial}\varepsilon \operatorname{id}_{E})\circ S^{-1}
=S\circ\widehat{\chi}\circ(\overline{\partial}\, \widehat{\otimes}_{\pi}\operatorname{id}_{E})
  \circ\widehat{\chi}^{-1}\circ S^{-1}.
\]
Now, we turn to the actual proof. Let $g\in\mathcal{EV}(\Omega,E)$.
The maps $\operatorname{id}_{E}\colon E\to E$ and 
$\overline{\partial}\colon\mathcal{EV}(\Omega)\to\mathcal{EV}(\Omega)$ are linear, continuous and surjective, 
the latter one by \prettyref{thm:scalar_CR_surjective}. Moreover, $E$ and $\mathcal{EV}(\Omega)$ are Fr\'echet spaces, 
so $\overline{\partial}\,\widehat{\otimes}_{\pi}\operatorname{id}_{E}$ is surjective by \cite[10.24 Satz, p.\ 255]{Kaballo}, 
i.e.\ there is $f \in\mathcal{EV}(\Omega)\widehat{\otimes}_{\pi}E$ such that 
$(\overline{\partial}\,\widehat{\otimes}_{\pi}\operatorname{id}_{E})(f)=(\widehat{\chi}^{-1}\circ S^{-1})(g)$. 
Then $(S\circ\widehat{\chi})(f)\in\mathcal{EV}(\Omega,E)$ and 
\[
 \overline{\partial}^{E}((S\circ\widehat{\chi})(f))
=(S\circ\widehat{\chi})((\overline{\partial}\, \widehat{\otimes}_{\pi}\operatorname{id}_{E})(f))
=(S\circ\widehat{\chi})((\widehat{\chi}^{-1}\circ S^{-1})(g))
=g.
\]
\end{proof}

%% file: density3.tex
In our last section we derive sufficient conditions for the very \textbf{w}eak \textbf{r}educibility of 
$\mathcal{EV}_{\overline{\partial}}(\Omega)$ and apply our main result to some examples.
We start with our conditions.

\begin{condWR}\label{cond:dense}
Let $\mathcal{V}:=(\nu_{n})_{n\in\N}$ be a (directed) family of continuous weights on an open set
$\Omega\subset\R^{2}$ and $(\Omega_{n})_{n\in\N}$ a family of non-empty
open sets such that $\Omega_{n}\neq \R^{2}$, $\Omega_{n}\subset\Omega_{n+1}$ for all $n\in\N$, 
$\d_{n,k}:=\d^{|\cdot|}(\Omega_{n},\partial\Omega_{k})>0$ for all $n,k\in\N$, $k>n$,
and $\Omega=\bigcup_{n\in\N} \Omega_{n}$.

(WR.1) For every $n\in\N$ let there be $g_{n}\in\mathcal{O}(\C)$ with $g_{n}(0)=1$ and $\N\ni I_{j}(n)> n$ for $j=1,2$ 
such that
\begin{enumerate}
\item [(a)] for every $\varepsilon>0$ there is a compact set $K\subset \overline{\Omega}_{n}$ with 
$\nu_{n}(x)\leq\varepsilon\nu_{I_{1}(n)}(x)$ for all $x\in\Omega_{n}\setminus K$.
\item [(b)] there is an open set $X_{I_{2}(n)}\subset\R^{2}\setminus \overline{\Omega}_{I_{2}(n)}$ such that
there are $R_{n},r_{n}\in\R$ with $0<2R_{n}<\d^{|\cdot|}(X_{I_{2}(n)},\Omega_{I_{2}(n)}):=\d_{X,I_{2}(n)}$ 
and $R_{n}<r_{n}<\d_{X,I_{2}(n)}-R_{n}$ as well as 
$A_{2}(\cdot,n)\colon X_{I_{2}(n)}+\mathbb{B}_{R_{n}}(0)\to (0,\infty)$, 
$A_{2}(\cdot,n)_{\mid X_{I_{2}(n)}}$ locally bounded, satisfying
\begin{equation}\label{pro.2}
\max\{|g_{n}(\zeta)|\nu_{I_{2}(n)}(z)\;|\;\zeta\in\R^{2},\,|\zeta-(z-x)|=r_{n}\}\leq A_{2}(x,n) 
\end{equation}
for all $z\in\Omega_{I_{2}(n)}$ and $x\in X_{I_{2}(n)}+\mathbb{B}_{R_{n}}(0)$.
\item [(c)] for every compact set $K\subset \R^{2}$ there is $A_{3}(n,K)>0$  with
\[
\int_{K}{\frac{|g_{n}(x-y)|\nu_{n}(x)}{|x-y|}\d y}\leq A_{3}(n,K),\quad x\in \Omega_{n}.
\]
\end{enumerate}

(WR.2) Let (WR.1a) be fulfilled. For every $n\in\N$ let there be $\N\ni I_{4}(n)>n$ and $A_{4}(n)>0$ such that
\begin{equation}\label{pro.4}
\int_{\Omega_{I_{4}(n)}}{\frac{|g_{I_{14}(n)}(x-y)|\nu_{p}(x)}{|x-y|\nu_{k}(y)}\d y}\leq A_{4}(n), \quad x\in \Omega_{p},
\end{equation}
for $(k,p)=(I_{4}(n),n)$ and $(k,p)=(I_{14}(n),I_{14}(n))$ where $I_{14}(n):=I_{1}(I_{4}(n))$.

(WR.3) Let (WR.1a), (WR.1b) and (WR.2) be fulfilled. For every $n\in\N$, every closed subset $M\subset \overline{\Omega}_{n}$ and every component $N$ of $M^{C}$ we have
\[
N\cap \overline{\Omega}_{n}^{C}\neq \varnothing\;\Rightarrow\; N\cap X_{I_{214}(n)}\neq \varnothing
\]
where $I_{214}(n):=I_{2}(I_{14}(n))$.
\end{condWR}

We use the same convention for $I_{j}$ as for $J_{i}$ (see \prettyref{conv:index}). Condition (WR.1a) appears in 
\cite[p.\ 67]{Bierstedt1975} under the name $(RU)$ (cf.\ \cite[Remark 3.4, p.\ 239]{kruse2018_2}) as well. 
(WR.1a) is used for approximation by compactly supported functions, (WR.1b) to control Cauchy estimates,
(WR.1c) as well as (WR.2) to guarantee that several kinds of convolutions of the fundamental solution $z\mapsto g_{n}(z)/(\pi z)$ 
of the $\overline{\partial}$-operator with certain functionals are well-defined and (WR.3) to control 
the support of an analytic distribution by using the identity theorem.

We begin with the proof that $(WR)$ is sufficient for very weak reducibility. 
The underlying idea of the proof is extracted from a proof of H\"ormander \cite[Theorem 4.4.5, p.\ 112]{H1} 
in a comparable situation for non-weighted $\mathcal{C}^{\infty}$-functions. 
The proof is split into several parts to enhance comprehensibility.

\begin{thm}\label{thm:dense_proj_lim}
Let $n\in\N$. Then the space $\pi_{I_{214}(n),n}(\mathcal{E}\nu_{I_{214}(n),\overline{\partial}}(\Omega_{I_{214}(n)}))$ is 
dense in $\pi_{I_{14}(n),n}(\mathcal{E}\nu_{I_{14}(n),\overline{\partial}}(\Omega_{I_{14}(n)}))$ w.r.t.\
$(|\cdot|_{n,m})_{m\in\N_{0}}$ if $(WR)$ is fulfilled. In particular, $\mathcal{EV}_{\overline{\partial}}(\Omega)$ is very weakly reduced if $(WR)$ is fulfilled.
\end{thm}

In order to gain access to the theory of distributions in our approach, we prove another density statement first 
whose proof is a modification of the one of \cite[Lemma 3.14, p.\ 242]{kruse2018_2}.

\begin{lem}\label{lem:dense_comp_supp}
Let $n\in \N$. Then the space
$\pi_{I_{1}(n),n}(\mathcal{D}(\Omega_{I_{1}(n)}))$ is dense in the space 
$\pi_{I_{1}(n),n}(\mathcal{E}\nu_{I_{1}(n)}(\Omega_{I_{1}(n)}))$ w.r.t.\ $(|\cdot|_{n, m})_{m\in\N_{0}}$ 
if $(WR.1a)$ is fulfilled.  
\end{lem}
\begin{proof}
Let $f\in\mathcal{E}\nu_{I_{1}}(\Omega_{I_{1}})$ and $\varepsilon>0$. Due to $(WR.1a)$ there is a compact set 
$K\subset\overline{\Omega}_{n}$ such that 
\begin{equation}\label{lem2.1}
\sup_{x \in \Omega_{n}\setminus K}\frac{\nu_{n}(x)}{\nu_{I_{1}}(x)}\leq\varepsilon
\end{equation}
where we use the convention $\sup_{x\in\varnothing}\tfrac{\nu_{n}(x)}{\nu_{I_{1}}(x)}:=-\infty$.
Like in the proof of \cite[Theorem 1.4.1, p.\ 25]{H1} we can find $\varphi\in\mathcal{D}(\Omega_{I_{1}})$, 
$0\leq \varphi\leq 1$, such that $\varphi=1$ near $K$ and
\begin{equation}\label{lem2.2}
|\partial^{\alpha}\varphi|\leq C_{\alpha}D^{-|\alpha|}
\end{equation}
for all $\alpha\in\N^{2}_{0}$ where $D:=\d_{n,I_{1}}/2$ and $C_{\alpha}>0$ is a constant 
only depending on $\alpha$. Then $\varphi f\in\mathcal{D}(\Omega_{I_{1}})$ and 
with $K_{0}:=\operatorname{supp}\varphi$ we have for $m\in\N_{0}$ by the Leibniz rule
\begin{flalign*}
&\hspace{0.37cm}|\varphi f -f|_{n,m}\\
&\leq\sup_{\substack{x \in \Omega_{n}\setminus K \\ \alpha\in\N^{2}_{0}, |\alpha| \leq m}}
 |\partial^{\alpha}( \varphi f)(x)|\nu_{n}(x)
 +\sup_{\substack{x \in \Omega_{n}\setminus K\\ \alpha\in\N^{2}_{0}, |\alpha| \leq m}}
 |\partial^{\alpha} f(x)|\nu_{n}(x)\\
&\leq \sup_{\substack{x \in (\Omega_{n}\setminus K)\cap K_{0}\\ \alpha\in\N^{2}_{0}, |\alpha| \leq m}}
 \bigl|\sum_{\gamma\leq \alpha}\dbinom{\alpha}{\gamma}\partial^{\alpha-\gamma} \varphi(x)
 \partial^{\gamma}f(x)\bigr|\nu_{n}(x)
 +\sup_{\substack{x \in \Omega_{n}\setminus K\\ \alpha\in\N^{2}_{0},|\alpha| \leq m}}
 |\partial^{\alpha} f(x)|\nu_{I_{1}}(x)\frac{\nu_{n}(x)}{\nu_{I_{1}}(x)}\\
&\underset{\mathclap{\eqref{lem2.1}}}{\leq}\; \sup_{\alpha\in\N^{2}_{0},|\alpha| \leq m}
 \sum_{\gamma\leq \alpha}\dbinom{\alpha}{\gamma}\sup_{x \in K_{0}}|\partial^{\alpha-\gamma} \varphi(x)| 
 \bigl(\sup_{\substack{x \in \Omega_{n}\setminus K\\ \beta\in\N^{2}_{0},|\beta| \leq m}}
 |\partial^{\beta}f(x)|\nu_{n}(x)\bigr)
 +\varepsilon |f|_{I_{1},m}\\
&\underset{\mathclap{\eqref{lem2.1},\eqref{lem2.2}}}{\leq}\quad\underbrace{\sup_{\alpha\in\N^{2}_{0},|\alpha|\leq m}
 \sum_{\gamma\leq\alpha}\dbinom{\alpha}{\gamma}C_{\alpha-\gamma}D^{-|\alpha-\gamma|}}_{:=C(m,D)}
 \varepsilon|f|_{I_{1},m}+\varepsilon |f|_{I_{1},m}
=(C(m,D)+1)|f|_{I_{1},m}\varepsilon
\end{flalign*}
where $C(m,D)$ is independent of $\varepsilon$, proving the density.
\end{proof}

The next lemma is devoted to a special fundamental solution of the $\overline{\partial}$-operator and its properties, namely to 
$E_{n}\colon \C\setminus \{0\}\to \C,\; E_{n}(z):=\frac{g_{n}(z)}{\pi z},$ with $g_{n}$ 
from $(WR)$.

\begin{lem}\label{lem:prepare_convolution} 
Let $n\in\N$ and $(WR.1b)$, $(WR.1c)$ and $(WR.2)$ be fulfilled.
\begin{enumerate}
	\item [a)] Then $\overline{\partial}T_{E_{n}}=\delta$ in the distributional sense.
	\item [b)] Let $x\in X_{I_{2}(n)}+\mathbb{B}_{R_{n}}(0)$ and $\alpha\in\N^{2}_{0}$. Then 
	$\partial^{\alpha}_{x}[E_{n}(\cdot-x)]\in
	\mathcal{E}\nu_{I_{2}(n),\overline{\partial}}(\Omega_{I_{2}(n)})$.
	\item [c)] Let $K\subset\R^{2}$ be a compact set and $m\in\N_{0}$. Then 
  \begin{equation}\label{lem3.1}
   |T_{E_{n}}\ast\psi|_{n,m}\leq \frac{A_{3}(n,K)}{\pi}\|\psi\|_{m},\quad \psi\in\mathcal{C}^{\infty}_{c}(K),
  \end{equation}	
  with the convolution from \eqref{distr.falt.}. 
  In particular, $T_{E_{n}}\ast\psi\in\mathcal{E}\nu_{n}(\Omega_{n})$.
	\item [d)]
   \begin{enumerate}
	 \item [(i)] There exists $\varphi\in\mathcal{C}^{\infty}(\R^{2})$, $0\leq\varphi\leq 1$, such that 
	 $\varphi=1$ near $\overline{\Omega}_{n}$ and $\varphi=0$ near $\Omega_{I_{4}(n)}^{C}$ plus
   \begin{equation}\label{lem3.2}
    |\partial^{\alpha}\varphi|\leq c_{\alpha}\d_{n,I_{4}(n)}^{-|\alpha|}
   \end{equation}
   for all $\alpha\in\N^{2}_{0}$ where $c_{\alpha}>0$ is a constant only depending on $\alpha$.
   \item [(ii)] Choose $\varphi\in\mathcal{C}^{\infty}(\R^{2})$ like in (i) and $m\in\N_{0}$. 
   Then there is $A_{5}=A_{5}(n,m)$ such that
   \begin{equation}\label{lem3.3}
    |T_{E_{I_{14}(n)}}\ast(\varphi f)|_{p,m}\leq A_{5}|f|_{k,m}, 
    \quad f\in\mathcal{E}\nu_{I_{14}(n)}(\Omega_{I_{14}(n)}),
   \end{equation}
   for $(k,p)=(I_{4}(n),n)$ and $(k,p)=(I_{14}(n),I_{14}(n))$ where the convolution is defined by 
   the right-hand side of \eqref{distr.falt.} and we set $\varphi f:=0$ outside $\Omega_{I_{14}(n)}$.
   In particular, $T_{E_{I_{14}(n)}}\ast(\varphi f)\in\mathcal{E}\nu_{I_{14}(n)}(\Omega_{I_{14}(n)})$.
   \end{enumerate}
\end{enumerate}
\end{lem}
\begin{proof}
$a)$ Let $\varphi\in\mathcal{D}(\C)$ and set $E_{0}(z):=\frac{1}{\pi z}$, $z\neq 0$. 
Using $g_{n}\in\mathcal{O}(\C)$, $g_{n}(0)=1$ and the fact that $T_{E_{0}}$ is a fundamental solution 
of the $\overline{\partial}$-operator by \cite[Eq.\ (3.1.12), p.\ 63]{H1}, we get
\begin{align*}
\langle \overline{\partial}T_{E_{n}}, \varphi\rangle &= -\langle T_{E_{n}}, \overline{\partial}\varphi\rangle 
 = -\langle T_{E_{0}}, g_{n}\overline{\partial}\varphi\rangle 
 = -\langle T_{E_{0}}, \overline{\partial}(g_{n}\varphi)\rangle\\
&= \langle \overline{\partial}T_{E_{0}}, g_{n}\varphi\rangle
 = \langle \delta, g_{n}\varphi\rangle
 =g_{n}(0)\varphi(0)=\varphi(0)=\langle \delta, \varphi\rangle.
\end{align*}

$b)$ Since $x\in (X_{I_{2}}+\mathbb{B}_{R_{n}}(0))\subset \Omega_{I_{2}}^{C}$, it follows 
$\partial^{\alpha}_{x}[E_{n}(\cdot-x)]\in\mathcal{O}(\Omega_{I_{2}})$.
Let $z\in \Omega_{I_{2}}$ and $\beta\in\N^{2}_{0}$. We get by the Cauchy inequality and $(WR.1b)$
\begin{align*}
|\partial^{\beta}_{z}\partial^{\alpha}_{x}[E_{n}(z-x)]|
&\underset{\mathclap{\eqref{lem1}}}{=} |i^{\alpha_{2}+\beta_{2}}(-1)^{|\alpha|}E_{n}^{(|\alpha+\beta|)}(z-x)|\\
&\leq\frac{|\alpha+\beta|!}{r_{n}^{|\alpha+\beta|}}\max_{|\zeta-(z-x)|=r_{n}}{|E_{n}(\zeta)|}\\
&\leq\frac{1}{\pi}\frac{|\alpha+\beta|!}{r_{n}^{|\alpha+\beta|}(\d_{X,I_{2}}-R_{n}-r_{n})}\max_{|\zeta-(z-x)|=r_{n}}
{|g_{n}(\zeta)|}
\end{align*}		
and hence
\begin{flalign*}
&\hspace{0.37cm}|\partial^{\alpha}_{x}[E_{n}(\cdot-x)]|_{I_{2},m}\\
&\underset{\eqref{pro.2}}{\leq}\frac{1}{\pi}\sup_{\beta\in\N^{2}_{0},|\beta|\leq m}
\frac{|\alpha+\beta|!}{r_{n}^{|\alpha+\beta|}(\d_{X,I_{2}}-R_{n}-r_{n})}A_{2}(x,n)<\infty.
\end{flalign*}

$c)$ By the definition of distributional convolution $T_{E_{n}}\ast\psi\in\mathcal{C}^{\infty}(\R^{2})$ 
and for $x\in\R^{2}$ and $\alpha\in\N^{2}_{0}$ the following inequalities hold
\begin{align*}
|\partial^{\alpha}(T_{E_{n}}\ast\psi)(x)|
&=\bigl|\int_{\R^{2}}{E_{n}(y)(\partial^{\alpha}\psi)(x-y)\d y}\bigr|
\leq \|\psi\|_{|\alpha|}\int_{x-K}{|E_{n}(y)|\d y}\\
&=\frac{1}{\pi}\|\psi\|_{|\alpha|}\int_{K}{\frac{|g_{n}(x-y)|}{|x-y|}\d y}
\end{align*}
and thus by $(WR.1c)$
\[
|T_{E_{n}}\ast\psi|_{n,m}
\leq\frac{1}{\pi}A_{3}(n,K)\|\psi \|_{m}.
\]

$d)$ The existence of $\varphi$ follows from the proof of \cite[Theorem 1.4.1, p.\ 25]{H1}. 
Now, let $x\in\Omega_{p}$ and $\alpha\in\N^{2}_{0}$. Then we have by $(WR.2)$, the Leibniz rule and due to 
$\operatorname{supp}\varphi\subset \Omega_{I_{4}}$ that
\begin{flalign*}
&\hspace{0.37cm}\bigl|\int_{\R^{2}}{E_{I_{14}}(y)\partial^{\alpha}(f\varphi)(x-y)\d y}\bigr|\\
&\leq\int_{x-\Omega_{I_{4}}}{|E_{I_{14}}(y)\partial^{\alpha}(f\varphi)(x-y)|\nu_{k}(x-y)\nu_{k}(x-y)^{-1}\d y}\\
&\leq \sup_{z \in \Omega_{I_{4}}}{|\partial^{\alpha}(f\varphi)(z)|\nu_{k}(z)}
 \int_{\Omega_{I_{4}}}{|E_{I_{14}}(x-y)|\nu_{k}(y)^{-1}\d y}\\
&\leq\phantom{\cdot}\sum_{\gamma\leq \alpha}{\dbinom{\alpha}{\gamma}\sup_{z \in\Omega_{I_{4}}}|\partial^{\alpha-\gamma} \varphi(z)|} 
 \sup_{\substack{w \in\Omega_{I_{4}}\\ \beta\in\N^{2}_{0}, |\beta| \leq |\alpha|}}
 |\partial^{\beta}f(w)|\nu_{k}(w)\\
&\phantom{\underset{\mathclap{\eqref{pro.4}}}{\leq}}\cdot \int_{\Omega_{I_{4}}}{|E_{I_{14}}(x-y)|\nu_{k}(y)^{-1}\d y}\\
&\underset{\mathclap{\eqref{lem3.2}}}{\leq}\phantom{\cdot}\underbrace{\sum_{\gamma\leq \alpha}{\dbinom{\alpha}{\gamma}
 c_{\alpha-\gamma}\d_{n,I_{4}}^{-|\alpha-\gamma|}}}_{=:C_{0}(\alpha,n)}
 \sup_{\substack{w \in\Omega_{I_{4}}\\\beta\in\N^{2}_{0},|\beta| \leq |\alpha|}}|\partial^{\beta}f(w)|\nu_{k}(w)\\
&\phantom{\underset{\mathclap{\eqref{pro.4}}}{\leq}}\cdot \int_{\Omega_{I_{4}}}{|E_{I_{14}}(x-y)|\nu_{k}(y)^{-1}\d y} \\
&\underset{\mathclap{\eqref{pro.4}}}{\leq}\frac{A_{4}(n)C_{0}(\alpha,n)}{\nu_{p}(x)}
 \sup_{\substack{w\in\Omega_{I_{4}}\\ \beta\in\N^{2}_{0},|\beta| \leq |\alpha|}}|\partial^{\beta}f(w)|\nu_{k}(w).
\end{flalign*}
Thus $T_{E_{I_{14}}}\ast(\varphi f)\in \mathcal{C}^{\infty}(\Omega_{p})$ and 
\[
\partial^{\alpha}(T_{E_{I_{14}}}\ast(\varphi f))(x)=\int_{\R^{2}}{E_{I_{14}}(y)\partial^{\alpha}(f\varphi)(x-y)\d y}
\]
by differentiation under the integral sign as well as 
\[
|T_{E_{I_{14}}}\ast(\varphi f)|_{p,m}
\leq A_{4}(n)\sup_{\alpha\in\N^{2}_{0},|\alpha|\leq m}C_{0}(\alpha,n)|f|_{k,m}
\]
for $(k,p)=(I_{4},n)$ and $(k,p)=(I_{14},I_{14})$.
\end{proof}

The next step is to define different kinds of convolutions and study their relations and properties, 
which shall be exploited in the proof of the density theorem.

\begin{lem}\label{lem:convolution}
Let $n\in\N$, $(WR.1b)$, $(WR.1c)$ and $(WR.2)$ be fulfilled and 
$w\in(\pi_{I_{14}(n),n}(\mathcal{E}\nu_{I_{14}(n)}(\Omega_{I_{14}(n)})), (|\cdot|_{n,m})_{m\in\N_{0}})'$.
\begin{enumerate}
 \item [a)] For $\psi\in\mathcal{D}(\R^{2})$ we define
  \[
   \langle w \ast_{1} T_{\check{E}_{I_{14}}}, \psi\rangle:=\langle w,(T_{E_{I_{14}}}\ast \psi)_{\mid\Omega_{n}}\rangle. 
	\]
 Then $w \ast_{1} T_{\check{E}_{I_{14}}}\in\mathcal{D}'(\R^{2})$.
 \item [b)] For $x\in X_{I_{214}}$ we define
	\[
	 (w \ast_{2} \check{E}_{I_{14}})(x):=\langle w,E_{I_{14}}(\cdot-x)_{\mid\Omega_{n}}\rangle.
	\]
 Then $w \ast_{2} \check{E}_{I_{14}}\in\mathcal{C}^{\infty}(X_{I_{214}})$ and for $\alpha\in\N^{2}_{0}$
  \begin{equation}\label{lem4.1}
    \partial^{\alpha}_{x}(w \ast_{2} \check{E}_{I_{14}})(x)
   =\langle w, \partial^{\alpha}_{x}[E_{I_{14}}(\cdot-x)]_{\mid\Omega_{n}}\rangle.
  \end{equation}
 \item [c)] For $\psi\in\mathcal{D}(\R^{2})$ with $\operatorname{supp}\psi\subset X_{I_{214}}$ 
 the preceding definitions of convolution are consistent, i.e.\
  \[
   \langle w \ast_{1} T_{\check{E}_{I_{14}}}, \psi\rangle = \langle T_{w \ast_{2} \check{E}_{I_{14}}},\psi\rangle.
  \]
 \item [d)] Choose $\varphi$ like in \prettyref{lem:prepare_convolution} d), let $m\in\N_{0}$ and 
 for  $f\in\mathcal{E}\nu_{I_{14}}(\Omega_{I_{14}})$ we define
	\[
	 \langle w\ast_{\varphi} T_{\check{E}_{I_{14}}},f\rangle:=\langle w, [T_{E_{I_{14}}}\ast (\varphi f)]_{\mid\Omega_{n}}\rangle .
	\]
 Then there exists a constant $A_{6}=A_{6}\left(w,n,m\right)>0$ such that
	\begin{equation}\label{lem4.0.1}
	 |\langle w\ast_{\varphi} T_{\check{E}_{I_{14}}}, f\rangle|\leq A_{6} |f|_{I_{4},m}.
	\end{equation}
\end{enumerate}
\end{lem}
\begin{proof}
$a)$ $w \ast_{1} T_{\check{E}_{I_{14}}}$ is defined by \prettyref{lem:prepare_convolution} c). 
Let $K\subset \R^{2}$ be compact. Since $w$ is continuous, there exist $C>0$ and $m\in\N_{0}$ such that
\begin{align*}
 |\langle w \ast_{1} T_{\check{E}_{I_{14}}},\psi\rangle|
&=|\langle w , (T_{E_{I_{14}}}\ast\psi)_{\mid\Omega_{n}}\rangle|
 \leq C |T_{E_{I_{14}}}\ast\psi|_{n,m}\\
&\leq C |T_{E_{I_{14}}}\ast\psi |_{I_{14},m}
 \underset{\eqref{lem3.1}}{\leq}\frac{CA_{3}(I_{14},K)}{\pi} \|\psi\|_{m}
\end{align*}
for all $\psi\in\mathcal{C}^{\infty}_{c}(K)$, thus $w \ast_{1} T_{\check{E}_{I_{14}}}\in\mathcal{D}'(\R^{2})$.

$b)$ $w \ast_{2} \check{E}_{I_{14}}$ and the right-hand side of \eqref{lem4.1} are defined 
by \prettyref{lem:prepare_convolution} b) since $I_{214}>I_{14}$. For $h\in\R$ with $0<|h|$ small enough and 
$x\in X_{I_{214}}$ we define
\[
 \psi_{h}(x)\colon \Omega_{I_{214}}\to\R^{2}, \; 
 \psi_{h}(x)[y]:=\frac{E_{I_{14}}(y-(x+he_{l}))-E_{I_{14}}(y-x)}{h}
\]
where $e_{l}$, $l=1,2$, is the $l$th canonical unit vector in $\R^{2}$. 
For $0<|h|<R_{I_{14}}$ we have $x+he_{l}\in X_{I_{214}}+\mathbb{B}_{R_{I_{14}}}(0)$ and so 
$E_{I_{14}}(\cdot-(x+he_{l}))\in\mathcal{E}\nu_{I_{214}}(\Omega_{I_{214}})$ 
by \prettyref{lem:prepare_convolution} b). 
Hence we get $\psi_{h}(x)\in\mathcal{E}\nu_{I_{214}}(\Omega_{I_{214}})$. 
The motivation for the definition of $\psi_{h}(x)$ comes from
\begin{flalign*}
&\hspace{0.37cm} \frac{(w \ast_{2} \check{E}_{I_{14}})(x+he_{l})-(w \ast_{2} \check{E}_{I_{14}})(x)}{h}\\
&=\bigl\langle w,\frac{E_{I_{14}}(\cdot-(x+he_{l}))-E_{I_{14}}(\cdot-x)}{h}_{\big{|}\Omega_{n}}\bigr\rangle
=\langle w,\psi_{h}(x)_{\mid\Omega_{n}}\rangle.
\end{flalign*}
So, if we show, that $\psi_{h}(x)$ converges to $\partial_{x_{l}}[E_{I_{14}}(\cdot-x)]$ in
$\mathcal{E}\nu_{_{I_{214}}}(\Omega_{I_{214}})$ as $h$ tends to $0$, 
we get, keeping $|\cdot|_{n,m}\leq|\cdot|_{I_{214},m}$ in mind,
\[
 \partial_{x_{l}}(w \ast_{2} \check{E}_{I_{14}})(x)
=\langle w, \partial_{x_{l}}[E_{I_{14}}(\cdot-x)]_{\mid\Omega_{n}}\rangle.
\]
Then the general statement follows by induction over the order $|\alpha|$.

Let $y\in \Omega_{I_{214}}$ and $\beta\in\N^{2}_{0}$. Since $|y-x|\geq \d_{X,I_{214}}>0$,
we get $0\notin \mathbb{B}_{\d_{X,I_{214}}}(y-x)$. 
Moreover, $R_{I_{14}}<\d_{X,I_{214}}$ by \prettyref{cond:dense} a)(ii) and so
\[
 |y-(x+he_{l})-(y-x)|=|h|<R_{I_{14}}<\d_{X,I_{214}}.
\]
Thus $y-(x+he_{l})\in \overline{\mathbb{B}_{|h|}(y-x)}\subset \mathbb{B}_{R_{I_{14}}}(y-x)$ 
and $0\notin\overline{\mathbb{B}_{|h|}(y-x)}$.
We write $E_{I_{14}}=(E_{I_{14},1},E_{I_{14},2})$ as a tuple of its coordinate functions. 
By the mean value theorem there exist $\zeta_{i}\in[y-(x+he_{l}),y-x]\subset\overline{\mathbb{B}_{|h|}(y-x)}$, $i=1,2$, 
where $[y-(x+he_{l}),y-x]$ denotes the line segment from $y-(x+he_{l})$ to $y-x$, such that
\begin{align*}
\partial^{\beta}_{y}\psi_{h}(x)[y]
&=\frac{(\partial^{\beta}E_{I_{14}})(y-(x+he_{l}))-(\partial^{\beta}E_{I_{14}})(y-x)}{h}\\
&=\frac{1}{h}
 \begin{pmatrix}
 \langle \nabla(\partial^{\beta}E_{I_{14},1})(\zeta_{1})|-he_{l}\rangle\\
 \langle \nabla(\partial^{\beta}E_{I_{14},2})(\zeta_{2})|-he_{l}\rangle
 \end{pmatrix}
=-\begin{pmatrix}
 \partial_{l}\partial^{\beta}E_{I_{14},1}(\zeta_{1})\\
 \partial_{l}\partial^{\beta}E_{I_{14},2}(\zeta_{2})
 \end{pmatrix},
\end{align*}
where $\nabla$ denotes the gradient, as well as $\zeta_{ii}\in[\zeta_{i},y-x]\subset\overline{\mathbb{B}_{|h|}(y-x)}$, $i=1,2$, such that
\begin{flalign}\label{lem4.2}
&\hspace{0.37cm}\partial^{\beta}_{y}\psi_{h}(x)[y]-\partial^{\beta}_{y}\partial_{x_{l}}[E_{I_{14}}(y-x)]\nonumber\\
&=-\begin{pmatrix}
 \partial_{l}\partial^{\beta}E_{I_{14},1}(\zeta_{1})\\
 \partial_{l}\partial^{\beta}E_{I_{14},2}(\zeta_{2})
 \end{pmatrix}
 -\partial^{\beta}(-\partial_{l}E_{I_{14}})(y-x)\nonumber\\
&=\begin{pmatrix}
 \langle \nabla(\partial_{l}\partial^{\beta}E_{I_{14},1})(\zeta_{11})|y-x-\zeta_{1}\rangle\\
 \langle \nabla(\partial_{l}\partial^{\beta}E_{I_{14},2})(\zeta_{22})|y-x-\zeta_{2}\rangle
 \end{pmatrix}.
\end{flalign}
Then
\begin{flalign}\label{lem4.3}
&\hspace{0.3cm}\left|\begin{pmatrix}
 \langle \nabla(\partial_{l}\partial^{\beta}E_{I_{14},1})(\zeta_{11})|y-x-\zeta_{1}\rangle\\
 \langle \nabla(\partial_{l}\partial^{\beta}E_{I_{14},2})(\zeta_{22})|y-x-\zeta_{2}\rangle
 \end{pmatrix}\right|\nonumber\\
&\leq|\langle \nabla(\partial_{l}\partial^{\beta}E_{I_{14},1})(\zeta_{11})|y-x-\zeta_{1}\rangle|
 +|\langle \nabla(\partial_{l}\partial^{\beta}E_{I_{14},2})(\zeta_{22})|y-x-\zeta_{2}\rangle|\nonumber\\
&\leq |\nabla(\partial_{l}\partial^{\beta}E_{I_{14},1})(\zeta_{11})| |y-x-\zeta_{1}|
 +|\nabla(\partial_{l}\partial^{\beta}E_{I_{14},2})(\zeta_{22})| |y-x-\zeta_{2}|\nonumber\\
&\leq \phantom{+}(|\partial_{1}\partial_{l}\partial^{\beta}E_{I_{14},1}(\zeta_{11})|
 +|\partial_{2}\partial_{l}\partial^{\beta}E_{I_{14},1}(\zeta_{11})|\nonumber\\
&\phantom{\leq}+|\partial_{1}\partial_{l}\partial^{\beta}E_{I_{14},2}(\zeta_{22})|
 +|\partial_{2}\partial_{l}\partial^{\beta}E_{I_{14},2}(\zeta_{22})|) 
  |h|\nonumber\\
&\leq \phantom{+}(|\partial_{1}\partial_{l}\partial^{\beta}E_{I_{14}}(\zeta_{11})|
 +|\partial_{2}\partial_{l}\partial^{\beta}E_{I_{14}}(\zeta_{11})|\nonumber\\
&\phantom{\leq}+|\partial_{1}\partial_{l}\partial^{\beta}E_{I_{14}}(\zeta_{22})|
 +|\partial_{2}\partial_{l}\partial^{\beta}E_{I_{14}}(\zeta_{22})|)
  |h|\nonumber\\
&\underset{\mathclap{\eqref{lem1}}}{=}2(|E_{I_{14}}^{(|\beta|+2)}(\zeta_{11})|
 +|E_{I_{14}}^{(|\beta|+2)}(\zeta_{22})|)|h|
\end{flalign}
is valid. 
By the choice $R_{I_{14}}<r_{I_{14}}<\d_{X,I_{214}}-R_{I_{14}}$ from $(WR.1b)$ we get 
due to Cauchy's integral formula
\begin{flalign}\label{lem4.4}
&\hspace{0.37cm} |E_{I_{14}}^{(|\beta|+2)}(\zeta_{ii})|\nonumber\\
&=\frac{(|\beta|+2)!}{2\pi }\bigl|\int_{\partial \mathbb{B}_{r_{I_{14}}}(y-x)}
 {\frac{E_{I_{14}}(\zeta)}{(\zeta-\zeta_{ii})^{|\beta|+3}}\d\zeta}\bigr|\nonumber\\
&\leq \frac{r_{I_{14}}(|\beta|+2)!}{(r_{I_{14}}-R_{I_{14}})^{|\beta|+3}}
 \max_{|\zeta-(y-x)|=r_{I_{14}}}{|\frac{g_{I_{14}}(\zeta)}{\pi \zeta}|}\nonumber\\
&\leq \underbrace{\frac{r_{I_{14}}(|\beta|+2)!}{\pi(r_{I_{14}}-R_{I_{14}})^{|\beta|+3}(\d_{X,I_{214}}-r_{I_{14}})}}_{=:C(n,|\beta|)}
 \max_{|\zeta-(y-x)|=r_{I_{14}}}{|g_{I_{14}}(\zeta)|}.
\end{flalign}
Hence by combining \eqref{lem4.2}, \eqref{lem4.3} and \eqref{lem4.4}, we have for $m\in\N_{0}$
\[
\hspace{0.65cm}|\psi_{h}(x)-\partial_{x_{l}}[E_{I_{14}}(\cdot-x)]|_{I_{214},m}
\underset{\eqref{pro.2}}{\leq} 4 \sup_{\beta\in\N^{2}_{0},|\beta|\leq m}C(n,|\beta|)A_{2}(x,I_{14})|h|\underset{h\to 0}{\to} 0.
\]
This means that $\psi_{h}(x)$ converges to $\partial_{x_{l}}[E_{I_{14}}(\cdot-x)]$ 
in $\mathcal{E}\nu_{I_{214}}(\Omega_{I_{214}})$ and so with respect 
to $(|\cdot|_{n,m})_{m\in\N_{0}}$ as well since $|\cdot|_{n,m}\leq|\cdot|_{I_{214},m}$.

$c)(i)$ For $h>0$ small enough we define
\[
S_{h}(\psi)\colon \Omega_{I_{14}}\to\R^2,\;
S_{h}(\psi)(y):=\sum_{m\in\Z^2}{E_{I_{14}}(y-mh)\psi(mh)h^2},
\]
where $E_{I_{14}}(0)\psi(mh)=E_{I_{14}}(0)0:=0$ if $mh\in\Omega_{I_{14}}$. 
The first part of the proof is to show that $S_{h}(\psi)$ converges to 
$T_{E_{I_{14}}}\ast\psi$ in $\mathcal{E}\nu_{I_{14}}(\Omega_{I_{14}})$ as $h$ tends to $0$.

Set $Q_{m}:=mh+[0,h]^{2}$ and let $N\subset X_{I_{214}}$ be compact.
Now, we define $M_{N,h}:=\{m\in\Z^{2}\; | \; Q_{m}\cap N\neq \varnothing\}$. Due to this definition we have
\begin{equation}\label{lem4.6}
\{m\in\Z^{2}\; | \; mh\in N\}\subset M_{N,h}
\end{equation}
and
\begin{equation}\label{lem4.8}
|M_{N,h}| 
\leq \left\lceil \frac{\operatorname{diam}(N)}{h}\right\rceil^{2}
\leq \bigl(\frac{\operatorname{diam}(N)}{h}+1\bigr)^{2}
\end{equation}
where $|M_{N,h}|$ denotes the cardinality of $M_{N,h}$, $\lceil x\rceil$ the ceiling of $x$ and 
$\operatorname{diam}(N)$ the diameter of $N$ w.r.t.\ $|\cdot|$. 
Let $0<h<\tfrac{1}{2\sqrt{2}}\d^{|\cdot|}(N, \partial X_{I_{214}})$. Then
\begin{equation}\label{lem4.9.0}
Q_{m}\subset (N+\overline{\mathbb{B}_{\frac{1}{2}\d^{|\cdot|}(N,\partial X_{I_{214}})}(0)})=:K\subset X_{I_{214}},
\quad m\in M_{N,h},
\end{equation}
as $\sqrt{2}h$ is the length of the diagonal of any cube $Q_{m}$.
Therefore we obtain for $y\in \Omega_{I_{14}} \subset\Omega_{I_{214}}$, $x\in Q_{m}$, $m\in M_{N,h}$ and 
$\beta\in\N^{2}_{0}$ analogously to the proof of \prettyref{lem:prepare_convolution} b) 
with the choice of $r_{I_{14}}$ from $(WR.1b)$
\begin{align}\label{lem4.10}
 |\partial^{\beta}_{y}[E_{I_{14}}(y-x)]|
&\leq\frac{|\beta|!}{\pi r_{I_{14}}^{|\beta|}}
 \max_{|\zeta-(y-x)|=r_{I_{14}}}{\frac{|g_{I_{14}}(\zeta)|}{|\zeta|}}\nonumber\\
&\leq\underbrace{\frac{|\beta|!}{\pi r_{I_{14}}^{|\beta|}(\d_{X,I_{214}}-r_{I_{14}})}}_{=:C_{1}(|\beta|,n)}
 \max_{|\zeta-(y-x)|=r_{I_{14}}}{|g_{I_{14}}(\zeta)|}\nonumber\\
&\underset{\mathclap{\eqref{pro.2}}}{\leq}C_{1}(|\beta|,n)\frac{A_{2}(x,I_{14})}{\nu_{I_{214}}(y)}.
\end{align}
Due to $(WR.1b)$ and \eqref{lem4.9.0} there is $C_{0}>0$, independent of $h$, 
such that for every $m\in M_{N,h}$ 
\begin{equation}\label{lem4.9.1}
  A_{2}(x,I_{14})\leq \sup_{z\in K}A_{2}(z,I_{14})\leq C_{0},\quad  x\in Q_{m}.
\end{equation}
Let $\psi\in\mathcal{C}^{\infty}_{c}(N)$ and $m_{0}\in\N_{0}$. Then we have
\begin{align*}
 |\partial^{\beta}_{y}S_{h}(\psi)(y)|\;&\underset{\mathclap{\eqref{lem4.6}}}{=}\;
 \bigl|\sum_{m\in M_{N,h}}{\partial^{\beta}_{y}[E_{I_{14}}(y-mh)]\psi(mh)h^2}\bigr|\\
&\underset{\mathclap{\eqref{lem4.10},\, mh\in Q_{m}}}{\leq} h^2C_{1}(|\beta|,n)\frac{1}{\nu_{I_{214}}(y)}
 \sum_{m\in M_{N,h}}A_{2}(mh,I_{14})|\psi(mh)|\\
&\underset{\mathclap{\eqref{lem4.8},\eqref{lem4.9.1}}}{\leq}C_{0}C_{1}(|\beta|,n)h^2
 \bigl(\frac{\operatorname{diam}(N)}{h}+1\bigr)^{2}\frac{1}{\nu_{I_{214}}(y)}\|\psi\|_{0}\\
&=C_{0}C_{1}(|\beta|,n)(\operatorname{diam}(N)+h)^{2}\frac{1}{\nu_{I_{214}}(y)}\|\psi\|_{0}
\end{align*}
and therefore
\begin{flalign*}
&\hspace{0.37cm} |S_{h}(\psi)|_{I_{14},m_{0}}\\
&\leq C_{0}\sup_{\beta\in\N^{2}_{0},|\beta|\leq m_{0}}{C_{1}(|\beta|,n)}(\operatorname{diam}(N)+h)^{2}
 \sup_{y\in \Omega_{I_{14}}}{\frac{\nu_{I_{14}}(y)}{\nu_{I_{214}}(y)}}\|\psi\|_{0}\\
&\leq C_{0}\sup_{\beta\in\N^{2}_{0},|\beta|\leq m_{0}}{C_{1}(|\beta|,n)}(\operatorname{diam}(N)+h)^{2}\|\psi\|_{0}
\end{flalign*}
bringing forth $S_{h}\left(\psi\right)\in\mathcal{E}\nu_{I_{14}}(\Omega_{I_{14}})$.
Further, the following equations hold
\begin{flalign}\label{lem4.12}
&\hspace{0.37cm}|\partial^{\beta}(S_{h}(\psi)-T_{E_{I_{14}}}\ast\psi)(y)|\nonumber\\
&=\bigl|\sum_{m\in M_{N,h}}\underbrace{{\partial^{\beta}_{y}[E_{I_{14}}(y-mh)]\psi(mh)h^2}}_{
   =\int_{Q_{m}}{\partial^{\beta}_{y}[E_{I_{14}}(y-mh)]\psi(mh)\d x}}
  -\underbrace{\int_{\R^2}{\partial^{\beta}_{y}[E_{I_{14}}(y-x)]\psi(x)\d x}}_{
   =\sum_{m\in M_{N,h}}{\int_{Q_{m}}{\partial^{\beta}_{y}[E_{I_{14}}(y-x)]\psi(x)\d x}}}\bigl|\nonumber\\
&=\bigl| \sum_{m\in M_{N,h}}\,{\int_{Q_{m}}{(\partial^{\beta}E_{I_{14}})(y-mh)\psi(mh)
  -(\partial^{\beta}E_{I_{14}})(y-x)\psi(x)\d x}}\bigr|\nonumber\\
&=\bigl|\sum_{m\in M_{N,h}}\,{\int_{Q_{m}}{[(\partial^{\beta}E_{I_{14}})(y-mh)-(\partial^{\beta}E_{I_{14}})(y-x)]\psi(mh)}}\nonumber\\
& \phantom{\bigl|\sum_{m\in M_{N,h}}\,{\int_{Q_{m}}}}\;+{{[\psi(mh)-\psi(x)](\partial^{\beta}E_{I_{14}})(y-x)\d x}}\bigr| .
\end{flalign}
The next steps are similar to the proof of $b)$. By the mean value theorem there exist 
$x_{0,i},x_{1,i}\in[x,mh]\subset Q_{m}$, $i=1,2$, such that for $\psi=(\psi_{1},\psi_{2})$
\begin{align}\label{lem4.13}
 |\psi(mh)-\psi(x)|
&=\left|
 \begin{pmatrix}
 \langle \nabla(\psi_{1})(x_{0,1})|mh-x\rangle\\
 \langle \nabla(\psi_{2})(x_{0,2})|mh-x\rangle
 \end{pmatrix}\right|
 \leq 4\|\psi\|_{1}|mh-x|\nonumber\\
&\leq 4\sqrt{2}h\|\psi\|_{1}
\end{align}
and
\begin{flalign}\label{lem4.14}
&\hspace{0.37cm}|(\partial^{\beta}E_{I_{14}})(y-mh)-(\partial^{\beta}E_{I_{14}})(y-x)|\nonumber\\
&=\left|-
 \begin{pmatrix}
 \langle \nabla(\partial^{\beta}E_{I_{14},1})(y-x_{1,1})|mh-x\rangle\\
 \langle \nabla(\partial^{\beta}E_{I_{14},2})(y-x_{1,2})|mh-x\rangle
 \end{pmatrix}\right|\nonumber\\
&\leq 2(E_{I_{14}}^{(|\beta|+1)}(y-x_{1,1})+E_{I_{14}}^{(|\beta|+1)}(y-x_{1,2}))|mh-x|\nonumber\\
&\underset{\mathclap{\eqref{lem4.10}}}{\leq}4\sqrt{2}hC_{1}(|\beta|+1,n)\frac{1}{\nu_{I_{214}}(y)}
 (A_{2}(x_{1,1},I_{14})+A_{2}(x_{1,2},I_{14}))\nonumber\\
&\underset{\mathclap{\eqref{lem4.9.1}}}{\leq} 4\sqrt{2}hC_{1}(|\beta|+1,n)\frac{2C_{0}}{\nu_{I_{214}}(y)}
\end{flalign}
analogously to \eqref{lem4.3}. Thus by combining \eqref{lem4.12}, \eqref{lem4.13} and \eqref{lem4.14}, we obtain
\begin{flalign*}
&\hspace{0.37cm}|\partial^{\beta}(S_{h}(\psi)-E_{I_{14}}\ast\psi)(y)|\\
&\leq\sum_{m\in M_{N,h}}\,\int_{Q_{m}}4\sqrt{2}h\bigl(C_{1}(|\beta|+1,n)\frac{2C_{0}}{\nu_{I_{214}}(y)}|\psi(mh)|\\
&\phantom{\leq\sum_{m\in M_{N,h}\,}\int_{Q_{m}}4\sqrt{2}h}\; +\|\psi\|_{1}|(\partial^{\beta}E_{I_{14}})(y-x)|\bigr)\d x\\
&\underset{\mathclap{\eqref{lem4.10},\eqref{lem4.9.1}}}{\leq}\quad\, \sum_{m\in M_{N,h}}{8\sqrt{2}C_{0}C_{2}(|\beta|,n)
 h\frac{1}{\nu_{I_{214}}(y)}(\|\psi\|_{0}+\|\psi\|_{1})\lambda(Q_{m})} \\
&\underset{\mathclap{\eqref{lem4.8}}}{\leq}8\sqrt{2}h^{3}\bigl(\frac{\operatorname{diam}(N)}{h}+1\bigr)^{2}
 C_{0}C_{2}(|\beta|,n)\frac{1}{\nu_{I_{214}}(y)}(\|\psi\|_{0}+\|\psi\|_{1})\\
&\leq 16\sqrt{2}(\operatorname{diam}(N)+h)^{2}hC_{0}C_{2}(|\beta|,n)\frac{1}{\nu_{I_{214}}(y)}\|\psi\|_{1}
\end{flalign*}
with $C_{2}(|\beta|,n):=\max\{C_{1}(|\beta|+1,n),C_{1}(|\beta|,n)\}$ and so for $m_{0}\in\N_{0}$
\begin{flalign*}
&\hspace{0.37cm}|S_{h}(\psi)-T_{E_{I_{14}}}\ast\psi|_{I_{14},m_{0}}\\
&\leq 16\sqrt{2}(\operatorname{diam}(N)+h)^{2}hC_{0}\sup_{\beta\in\N^{2}_{0},|\beta|\leq m_{0}}{C_{2}(|\beta|,n)}
 \sup_{y\in \Omega_{I_{14}}}{\frac{\nu_{I_{14}}(y)}{\nu_{I_{214}}(y)}}\|\psi\|_{1}\\
&\leq 16\sqrt{2}C_{0}\sup_{\beta\in\N^{2}_{0},|\beta|\leq m_{0}}
 {C_{2}(|\beta|,n)}\|\psi\|_{1}(\operatorname{diam}(N)+h)^{2}h
\underset{h\to 0}{\to} 0,
\end{flalign*}
proving the convergence of $S_{h}(\psi)$ to $T_{E_{I_{14}}}\ast\psi$ in $\mathcal{E}\nu_{I_{14}}(\Omega_{I_{14}})$ 
and hence with respect to $(|\cdot|_{n,m_{0}})_{m_{0}\in\N_{0}}$ as well.

$(ii)$ The next part of the proof is to show that
\[
 \lim_{h\to 0}\sum_{m\in M_{N,h}} (w\ast_{2}\check{E}_{I_{14}})(mh)\psi(mh)h^{2}
=\int_{\R^{2}}{(w\ast_{2}\check{E}_{I_{14}})(x)\psi(x)\d x}.
\]
Let $0<h<\tfrac{1}{2\sqrt{2}}\d^{|\cdot|}(N, \partial X_{I_{214}})$. We begin with
\begin{flalign}\label{lem4.16}
&\hspace{0.37cm}\bigl|\sum_{m\in M_{N,h}} (w\ast_{2}\check{E}_{I_{14}})(mh)\psi(mh)h^{2}
 -\int_{\R^2}{(w\ast_{2}\check{E}_{I_{14}})(x)\psi(x)\d x}\bigr|\nonumber\\
&=\bigl|\sum_{m\in M_{N,h}}\,\int_{Q_{m}}{ (w\ast_{2}\check{E}_{I_{14}})(mh)\psi(mh)
 -(w\ast_{2}\check{E}_{I_{14}})(x)\psi(x)\d x}\bigr|\nonumber\\
&=\bigl|\sum_{m\in M_{N,h}}\,{\int_{Q_{m}}
 {[(w\ast_{2}\check{E}_{I_{14}})(mh)-(w\ast_{2}\check{E}_{I_{14}})(x) ]\psi(mh)}}\nonumber\\
&\phantom{\sum_{m\in M_{N,h}}\,\int_{Q_{m}}}\;\;+{{[\psi(mh)-\psi(x)](w\ast_{2}\check{E}_{I_{14}})(x)\d x}}\bigr|.
\end{flalign}
Again, by the mean value theorem there exist $x_{0,i},\;x_{1,i}\in[x,mh]\subset Q_{m}$, $i=1,2$, for 
$x\in Q_{m}=mh+[0,h]^{2}$ such that 
\begin{equation}\label{lem4.17}
 |\psi(mh)-\psi(x)|
=\left|
 \begin{pmatrix}
 \langle \nabla(\psi_{1})(x_{0,1})|mh-x\rangle\\
 \langle \nabla(\psi_{2})(x_{0,2})|mh-x\rangle
 \end{pmatrix}\right|
\leq 4\sqrt{2}h\|\psi\|_{1}
\end{equation}
and for $w\ast_{2}\check{E}_{I_{14}}=((w\ast_{2}\check{E}_{I_{14}})_{1},(w\ast_{2}\check{E}_{I_{14}})_{2})$, 
taking account of \eqref{lem4.9.0} and part b), 
\begin{flalign}\label{lem4.18}
&\hspace{0.37cm}|(w\ast_{2}\check{E}_{I_{14}})(mh)-(w\ast_{2}\check{E}_{I_{14}})(x)|\nonumber\\
&=\left|
 \begin{pmatrix}
 \langle \nabla((w\ast_{2}\check{E}_{I_{14}})_{1})(x_{1,1})|mh-x\rangle\\
 \langle \nabla((w\ast_{2}\check{E}_{I_{14}})_{2})(x_{1,2})|mh-x\rangle
 \end{pmatrix}\right|\nonumber\\
&\leq(|\nabla((w\ast_{2}\check{E}_{I_{14}})_{1})(x_{1,1})|
 +|\nabla((w\ast_{2}\check{E}_{I_{14}})_{2})(x_{1,2})|)\sqrt{2}h\nonumber\\
&\leq \underbrace{(\|\nabla((w\ast_{2}\check{E}_{I_{14}})_{1})\|_{K}
 +\|\nabla((w\ast_{2}\check{E}_{I_{14}})_{2})\|_{K})}_{=:C_{3}<\infty}\sqrt{2}h
\end{flalign}
where we used $x_{1,i}\in Q_{m}$, $m\in M_{N,h}$, in the last inequality. 
Due to \eqref{lem4.16}, \eqref{lem4.17} and \eqref{lem4.18} we gain
\begin{flalign*}
&\hspace{0.37cm}\bigl|\sum_{m\in M_{N,h}}(w\ast_{2}\check{E}_{I_{14}})(mh)\psi(mh)h^{2}
 -\int_{\R^2}{(w\ast_{2}\check{E}_{I_{14}})(x)\psi(x)\d x}\bigr|\nonumber\\
&\leq \sum_{m\in M_{N,h}}(C_{3}\sqrt{2}h\|\psi\|_{0}+4\sqrt{2}h\|\psi\|_{1}\|w\ast_{2}\check{E}_{I_{14}}\|_{K})h^{2}\\
&\underset{\mathclap{\eqref{lem4.8}}}{\leq}(C_{3}\sqrt{2}\|\psi\|_{0}+4\sqrt{2}\|\psi\|_{1}\|
 w\ast_{2}\check{E}_{I_{14}}\|_{K})(\operatorname{diam}(N)+h)^{2}h
\underset{h\to 0}{\to} 0.
\end{flalign*}

$(iii)$ Merging $(i)$ and $(ii)$, we get for $\psi\in\mathcal{C}^{\infty}_{c}(N)$
\begin{align*}
 \langle w \ast_{1} T_{\check{E}_{I_{14}}}, \psi\rangle
&=\langle w , (T_{E_{I_{14}}}\ast\psi)_{\mid\Omega_{n}}\rangle
 \underset{(i)}{=}\lim_{h\to 0}\langle w , S_{h}(\psi)_{\mid\Omega_{n}}\rangle\\
&=\lim_{h\to 0}\langle w , \sum_{m\in M_{N,h}}{E_{I_{14}}(\cdot-mh)_{\mid\Omega_{n}}\psi(mh)h^2}\rangle\\
&\underset{\mathclap{\eqref{lem4.8}}}{=}\; \lim_{h\to 0}\sum_{m\in M_{N,h}}
 \underbrace{\langle w ,E_{I_{14}}(\cdot-mh)_{\mid\Omega_{n}}\rangle}_{
  =(w\ast_{2}\check{E}_{I_{14}})(mh)}\psi(mh)h^2\\
&\underset{\mathclap{(ii)}}{=}\;\int_{\R^2}{(w\ast_{2}\check{E}_{I_{14}})(x)\psi(x)\d x}
=\langle T_{w\ast_{2}\check{E}_{I_{14}}},\psi\rangle.
\end{align*}

$d)$ $w\ast_{\varphi} T_{\check{E}_{I_{14}}}$ is defined by \prettyref{lem:prepare_convolution} d). 
Because $w$ is continuous, there exist $C_{4}>0$ and $m\in\N_{0}$ such that
\begin{align*}
 |\langle w\ast_{\varphi} T_{\check{E}_{I_{14}}},f \rangle|
&=|\langle w, [T_{E_{I_{14}}}\ast(\varphi f)]_{\mid\Omega_{n}}\rangle|
\leq C_{4}|T_{E_{I_{14}}}\ast(\varphi f)|_{n,m}\\
&\underset{\mathclap{\eqref{lem3.3},\, k=I_{4},\,p=n}}{\leq}\;C_{4}A_{5}|f|_{I_{4},m}. 
\end{align*}
\end{proof}

\begin{lem}\label{lem:dense_component}
Let $n\in\N$, $w\in (\pi_{I_{14},n}(\mathcal{E}\nu_{I_{14}}(\Omega_{I_{14}})), (|\cdot|_{n,m})_{m\in\N_{0}})'$ 
and $(WR)$ be fulfilled.
If $w_{\mid\pi_{I_{214},n}(\mathcal{E}\nu_{I_{214},\overline{\partial}}(\Omega_{I_{214}}))}=0$, 
then $\operatorname{supp}(w \ast_{1} T_{\check{E}_{I_{14}}})\subset\overline{\Omega}_{n}$ 
where the support is meant in the distributional sense.
\end{lem}
\begin{proof}
$(i)$ Let $\psi\in\mathcal{C}^{\infty}_{c}(N)$ where $N\subset \R^{2}$ is compact. 
The set $K:=\overline{\Omega}_{I_{14}}\cap N$ is compactly contained in $\Omega$ and we have for $m\in\N_{0}$
\begin{align*}
 |\psi|_{I_{14},m}
&=\sup_{\substack{x\in \Omega_{I_{14}}\\ \beta\in\N^{2}_{0}, |\beta|\leq m}}
 |\partial^{\beta}\psi(x)|\nu_{I_{14}}(x)
 \leq \|\nu_{I_{14}}\|_{K}\sup_{\substack{x\in \R^{2}\\ \beta\in\N^{2}_{0},|\beta|\leq m}}|\partial^{\beta}\psi(x)|\\
&=\|\nu_{I_{14}}\|_{K}\|\psi\|_{m}<\infty,
\end{align*}
hence $\psi_{\mid\Omega_{I_{14}}}\in\mathcal{E}\nu_{I_{14}}(\Omega_{I_{14}})$. Now, we define
\[
w_{0}\colon \mathcal{D}(\R^{2})\to\mathcal{D}(\R^{2}),\;
w_{0}(\psi):=w(\pi_{I_{14},n}(\psi_{\mid\Omega_{I_{14}}}))=w(\psi_{\mid\Omega_{n}}).
\]
Then we obtain by the assumptions on $w$ that there exist $m\in\N_{0}$ and $C>0$ such that
\[
 |w_{0}(\psi)|
=|w(\psi_{\mid\Omega_{n}})|
\leq C |\psi|_{n,m} \leq C |\psi|_{I_{14},m}
\leq C \|\nu_{I_{14}}\|_{K}\|\psi\|_{m},
\]
for all $\psi\in\mathcal{C}^{\infty}_{c}(N)$ and therefore $w_{0}\in\mathcal{D}'(\R^{2})$ 
as well as $\operatorname{supp} w_{0}\subset \overline{\Omega}_{n}$.

$(ii)$ Let $\psi\in\mathcal{D}(\R^{2})$. Then we get
\begin{align*}
 \langle \overline{\partial}(w\ast_{1}T_{\check{E}_{I_{14}}}), \psi\rangle
&\;\underset{\mathclap{\ref{lem:convolution}\,a)}}{=}\;\langle w\ast_{1}T_{\check{E}_{I_{14}}}, -\overline{\partial}\psi\rangle
 =-\langle w, (T_{E_{I_{14}}}\ast\overline{\partial}\psi)_{\mid\Omega_{n}}\rangle\\
&\;=-\langle w, (\overline{\partial}T_{E_{I_{14}}}\ast\psi)_{\mid\Omega_{n}}\rangle
 \underset{\ref{lem:prepare_convolution}\,a)}{=}-\langle w, (\delta\ast\psi)_{\mid\Omega_{n}}\rangle\\
&\;=-\langle w, \psi_{\mid\Omega_{n}}\rangle\underset{(i)}{=}-\langle w_{0},\psi\rangle,
\end{align*}
thus $\overline{\partial}(w\ast_{1}T_{\check{E}_{I_{14}}})=-w_{0}$ and so
$\overline{\partial}(w\ast_{1}T_{\check{E}_{I_{14}}})=0$ 
on $\mathcal{D}(\R^{2}\setminus\operatorname{supp} w_{0})$ due to \cite[Theorem 2.2.1, p.\ 41]{H1}. 
Hence, by virtue of the ellipticity of the $\overline{\partial}$-operator, 
it exists $u\in\mathcal{O}(\C\setminus\operatorname{supp} w_{0})$ such that 
$T_{u}=w\ast_{1}T_{\check{E}_{I_{14}}}$ (see \cite[Theorem 11.1.1, p.\ 61]{H2}). 
By $(i)$ we have $\operatorname{supp} w_{0}\subset \overline{\Omega}_{n}$ and 
therefore we get $X_{I_{214}}\subset(\operatorname{supp} w_{0})^{C}$ and 
thus $\mathcal{D}(X_{I_{214}})\subset\mathcal{D}((\operatorname{supp} w_{0})^{C})$. 
It follows by \prettyref{lem:convolution} c) that
\[
T_{u}=w\ast_{1}T_{\check{E}_{I_{14}}}=T_{w\ast_{2}\check{E}_{I_{14}}}
\]
on $\mathcal{D}(X_{I_{214}})$, implying $u=w\ast_{2}\check{E}_{I_{14}}$ on $X_{I_{214}}$ by \prettyref{lem:convolution} b). 
This means we have for every $x\in X_{I_{214}}$ and $\alpha\in\N^{2}_{0}$
\begin{align*}
 u^{(|\alpha|)}(x)
&=(w\ast_{2}\check{E}_{I_{14}})^{(|\alpha|)}(x)
\underset{\eqref{lem1}}{=}i^{-\alpha_{2}}\partial^{\alpha}(w\ast_{2}\check{E}_{I_{14}})(x)\\
&\underset{\mathclap{\eqref{lem4.1}}}{=}i^{-\alpha_{2}}\langle w, \partial^{\alpha}_{x} [E_{I_{14}}(\cdot-x)]_{\mid\Omega_{n}}\rangle
 \underset{\ref{lem:prepare_convolution}\,b)}{=} 0
\end{align*}
by the assumptions on $w$. Hence $u=0$ in every component $N$ of $(\operatorname{supp} w_{0})^{C}$ 
with $N\cap X_{I_{214}}\neq\varnothing$ by the identity theorem. 
Denote by $N_{i}$, $i\in I$, the components of $(\operatorname{supp} w_{0})^{C}$ and let 
$I_{0}:=\{i\in I\;| \; N_{i}\cap \overline{\Omega}_n^{C}\neq\varnothing\}$. 
Due to $(WR.3)$ with $M:=\operatorname{supp} w_{0}$ we get $u=0$ on
\[
\bigcup_{i\in I_{0}}{N_{i}}\supset\bigl(\bigcup\limits_{i\in I_{0}}{N_{i}}\bigr)\cap\overline{\Omega}_{n}^{C}
=\bigl(\bigcup_{i\in I}{N_{i}}\bigr)\cap\overline{\Omega}_{n}^{C}
=(\operatorname{supp} w_{0})^{C}\cap\overline{\Omega}_{n}^{C}=\overline{\Omega}_{n}^{C}.
\]
Since $T_{u}=w\ast_{1}T_{\check{E}_{I_{14}}}$ on $\mathcal{D}((\operatorname{supp} w_{0})^{C})$, 
we conclude $\operatorname{supp}({w \ast_{1} T_{\check{E}_{I_{14}}}})\subset\overline{\Omega}_{n}$.
\end{proof}

Now, we are finally able to prove that $(WR)$ implies very weak reducibility.

\begin{proof}[\textbf{Proof of \prettyref{thm:dense_proj_lim} }] 
Set $G:=(\pi_{I_{14},n}(\mathcal{E}\nu_{I_{14},\overline{\partial}}(\Omega_{I_{14}})),(|\cdot|_{n,m})_{m\in\N_{0}})$ 
and $F:=\pi_{I_{214},n}(\mathcal{E}\nu_{I_{214},\overline{\partial}}(\Omega_{I_{214}}))\subset G$. 
Further, let $\widetilde{w}\in F^{\circ}:=\{y\in G'\;|\;\forall\;f\in F:\;y(f)=0\}$.
The space $H:=(\pi_{I_{14},n}(\mathcal{E}\nu_{I_{14}}(\Omega_{I_{14}})),(|\cdot|_{n,m})_{m\in\N_{0}})$ 
is a locally convex Hausdorff space and by the Hahn-Banach theorem exists 
$w\in H'$ such that $w_{\mid G}=\widetilde{w}$.

Let $f\in\mathcal{E}\nu_{I_{14},\overline{\partial}}(\Omega_{I_{14}})$ and 
$\varphi$ like in \prettyref{lem:prepare_convolution} d). 
By \prettyref{lem:dense_comp_supp} there exists a sequence 
$(\psi_{l})_{l\in\N}$ in $\mathcal{C}^{\infty}_{c}(\Omega_{I_{14}})$ which
converges to $f$ with respect to $(|\cdot|_{I_{4},m})_{m\in\N_{0}}$ and thus $(\overline{\partial}\psi_{l})_{l\in\N}$ 
to $\overline{\partial}f=0$ as well since
\[
 \overline{\partial}\colon \mathcal{E}\nu_{I_{4}}(\Omega_{I_{4}})\to\mathcal{E}\nu_{I_{4}}(\Omega_{I_{4}})
\]
is continuous. Therefore we obtain
\begin{align*}
 \langle \widetilde{w},\pi_{I_{14},n}(f)\rangle 
&=\langle \widetilde{w},f_{\mid\Omega_{n}}\rangle 
 =\langle w,f_{\mid\Omega_{n}}\rangle 
 \underset{n<I_{4}}{=} \lim_{l\to\infty}\langle w,{\psi_{l}}_{\mid\Omega_{n}}\rangle
 =\lim_{l\to\infty}\langle w,(\delta\ast\psi_{l})_{\mid\Omega_{n}}\rangle\\
&\underset{\mathclap{\ref{lem:prepare_convolution}\,a)}}{=}\;\;
 \lim_{l\to\infty}\langle w,(T_{E_{I_{14}}}\ast\overline{\partial}\psi_{l})_{\mid\Omega_{n}}\rangle
 =\lim_{l\to\infty}\langle w\ast_{1} T_{\check{E}_{I_{14}}},\overline{\partial}\psi_{l}\rangle\\
&\underset{\mathclap{\ref{lem:dense_component}}}{=}\;
 \lim_{l\to\infty}\langle w\ast_{1} T_{\check{E}_{I_{14}}},\varphi\overline{\partial}\psi_{l}\rangle
 =\lim_{l\to\infty}\langle w ,(T_{E_{I_{14}}}\ast\varphi\overline{\partial}\psi_{l})_{\mid\Omega_{n}}\rangle\\
&=\lim_{l\to\infty}\langle w\ast_{\varphi} T_{\check{E}_{I_{14}}},\overline{\partial}\psi_{l}\rangle
 \underset{\eqref{lem4.0.1}}{=}\langle w\ast_{\varphi} T_{\check{E}_{I_{14}}},\overline{\partial}f\rangle
 =0,
\end{align*}
so $\widetilde{w}=0$, yielding the statement due to the bipolar theorem. In particular, it follows from the choice 
$\iota_{1}(n):=I_{14}(n)$ and $\iota_{2}(n):=I_{214}(n)$ that $\mathcal{EV}_{\overline{\partial}}(\Omega)$ is very weakly reduced.
\end{proof}

The results obtained so far give rise to the following corollary of our main result. 

\begin{cor}\label{cor:CR_surjective_WR}
Let $(PN)$ with $\psi_{n}(z):=(1+|z|^{2})^{-2}$, $z\in\Omega$, and $(WR)$ with $I_{214}(n)\geq I_{14}(n+1)$ be fulfilled and $-\ln\nu_{n}$ be subharmonic on $\Omega$ for every $n\in\N$.
If $E$ is a Fr\'echet space over $\C$, then
\[
\overline{\partial}^{E}\colon \mathcal{EV}(\Omega,E)\to\mathcal{EV}(\Omega,E)
\]
is surjective.
\end{cor}
\begin{proof}
It follows from \prettyref{thm:dense_proj_lim} that $\mathcal{EV}_{\overline{\partial}}(\Omega)$ 
is very weakly reduced with 
$\iota_{1}(n):=I_{14}(n)$ and $\iota_{2}(n):=I_{214}(n)$ for $n\in\N$. 
Thus \prettyref{cor:frechet_CR_surjective} yields our statement.
\end{proof} 

\begin{exa}\label{ex:families_of_weights_2}
Let $\Omega\subset\C$ be a non-empty open set and $(\Omega_{n})_{n\in\N}$ a family of open sets such that
\begin{enumerate}
\item [(i)] $\Omega_{n}:=\{z\in \Omega\;|\;|\im(z)|<n\;\text{and}\;
 \d^{|\cdot|}(\{z\},\partial \Omega)>1/n \}$ for all $n\in\N$.
\item [(ii)] $\Omega_{n}:=\mathring K_{n}$ for all $n\in\N$ where 
 $K_{n}:=\overline{\mathbb{B}_{n}(0)}\cap\{z\in\Omega\;|\; \d^{|\cdot|}(\{z\},\partial\Omega)\geq 1/n\}$. 
\end{enumerate}
The following families $\mathcal{V}:=(\nu_{n})_{n\in\N}$ of continuous weight functions fulfil the assumptions of 
\prettyref{cor:CR_surjective_WR}:
\begin{enumerate}
\item [a)] Let $(a_{n})_{n\in\N}$ be strictly increasing such that $a_{n}\leq 0$ for all $n\in\N$ and 
\[
 \nu_{n}\colon\Omega\to (0,\infty),\;\nu_{n}(z):=e^{a_{n}|z|^{\gamma}},
\]
for some $0<\gamma\leq 1$ with $(\Omega_{n})_{n\in\N}$ from (i).
\item [b)] $\nu_{n}(z):=1$, $z\in\Omega$, with $(\Omega_{n})_{n\in\N}$ from (ii).
\end{enumerate}
\end{exa}
\begin{proof}
For each family $(\Omega_{n})_{n\in\N}$ in (i) and (ii) it holds that $\Omega_{n}\neq\C$ and there is $N\in\N_{0}$ 
such that $\Omega_{n}\neq\varnothing$ for all $n\geq N$. 
Hence we assume w.l.o.g.\ that $\Omega_{n}\neq \varnothing$ for every $n\in\N$ in what follows. 
In all the examples $(PN)$ is fulfilled for $\psi_{n}(z):=(1+|z|^{2})^{-2}$ 
by \prettyref{ex:families_of_weights_1} for all $q\in\N$. 
Further, we choose $I_{j}(n):=2n$ for $j=1,2,4$ and define the open set $X_{I_{2}(n)}:=\overline{\Omega}_{4n}^{C}$. 
Then we have
 \[
  I_{214}(n)=8n\geq 4n+4=I_{14}(n+1),\quad n\in\N.
 \]
The function $-\ln \nu_{n}$ is subharmonic on $\Omega$ for the considered weights by 
\cite[Corollary 1.6.6, p.\ 18]{H3} and \cite[Theorem 1.6.7, p.\ 18]{H3} since the function $z\mapsto z$ 
is holomorphic 
and $-a_{n}\geq 0$. 
Furthermore, we have $\d_{n,k}=|1/n-1/k|$ if $\partial \Omega\neq\varnothing$ and $\d_{n,k}=|n-k|$ if $\Omega=\C$ in (i) 
as well as $\d_{n,k}\geq |1/n-1/k|$ if $\partial \Omega\neq\varnothing$ and $\d_{n,k}=|n-k|$ if $\Omega=\C$ in (ii).
\begin{enumerate}
 \item [a)] $(WR.1a)$: The choice $K:=\overline{\Omega}_{n}$, if $\Omega_{n}$ is bounded, and 
  \[
   K:=\overline{\Omega}_{n}\cap\{z\in\C\;|\; |\re(z)|\leq \max(0,\ln(\varepsilon)/(a_{n}-a_{2n}))^{1/\gamma}+n\},
  \]
 if $\Omega_{n}$ is unbounded, guarantees that this condition is fulfilled.
 
 $(WR.1b)$: We have $\d_{X,I_{2}}=1/(2n)$ if $\partial\Omega\neq\varnothing$
 and $\d_{X,I_{2}}=2n$ if $\Omega=\C$ for $(\Omega_{n})_{n\in\N}$ from (i). 
 We choose $g_{n}\colon\C\to\C$, $g_{n}(z):=\exp(-z^2)$, 
 as well as $r_{n}:=1/(4n)$ and $R_{n}:=1/(6n)$ for $n\in\N$. 
 Let $z\in\Omega_{I_{2}(n)}$ and $x\in X_{I_{2}(n)}+\mathbb{B}_{R_{n}}(0)$. For $\zeta=\zeta_{1}+i\zeta_{2}\in\C$ 
 with $|\zeta-(z-x)|=r_{n}$ we have 
 \begin{align*}
 |g_{n}(\zeta)|\nu_{I_{2}(n)}(z)&=e^{-\re(\zeta^{2})}e^{a_{2n}|z|^{\gamma}}\leq e^{-\zeta_{1}^{2}+\zeta_{2}^{2}}
  \leq e^{(r_{n}+|z_{2}|+|x_{2}|)^{2}}e^{-\zeta_{1}^{2}}\\
 &\leq e^{(r_{n}+2n+|x_{2}|)^{2}}=:A_{2}(x,n)
 \end{align*}
 and observe that $A_{2}(\cdot,n)$ is continuous and thus locally bounded on $X_{I_{2}(n)}$.
 
 $(WR.1c)$: Let $K\subset\C$ be compact and $x=x_{1}+ix_{2}\in\Omega_{n}$. Then there 
 is $b>0$ such that $|y|\leq b$ for all $y=y_{1}+iy_{2}\in K$ and from polar coordinates and Fubini's theorem 
 follows that
 \begin{flalign*}
 &\hspace{0.37cm}\int_{K}\frac{|g_{n}(x-y)|}{|x-y|}\d y\\
 &=\int_{K}\frac{e^{-\re((x-y)^{2})}}{|x-y|}\d y
  \leq \int_{\mathbb{B}_{1}(x)}\frac{e^{-\re((x-y)^{2})}}{|x-y|}\d y
  +\int_{K\setminus \mathbb{B}_{1}(x)}\frac{e^{-\re((x-y)^{2})}}{|x-y|}\d y\\
 &\leq \int_{0}^{2\pi}\int_{0}^{1}\frac{e^{-r^{2}\cos(2\varphi)}}{r}r\d r\d\varphi
  + \int_{K\setminus \mathbb{B}_{1}(x)}e^{-\re((x-y)^{2})}\d y\\
 &\leq 2\pi e+\int_{-b}^{b}e^{(x_{2}-y_{2})^{2}}\d y_{2}\int_{\R}e^{-(x_{1}-y_{1})^{2}}\d y_{1}
  \leq 2\pi e+2be^{(|x_{2}|+b)^{2}}\int_{\R}e^{-y_{1}^{2}}\d y_{1}\\
 &=2\pi e+2\sqrt{\pi}be^{(|x_{2}|+b)^{2}}
  \leq 2\pi e+2\sqrt{\pi}be^{(n+b)^{2}}.
 \end{flalign*}
 We conclude that $(WR.1c)$ holds since $\nu_{n}\leq 1$.
 
 $(WR.2)$: Let $p,k\in\N$ with $p\leq k$. For all $x=x_{1}+ix_{2}\in\Omega_{p}$ 
 and $y=y_{1}+iy_{2}\in\Omega_{I_{4}(n)}$ we note that
 \begin{align*}
 a_{p}|x|^{\gamma}-a_{k}|y|^{\gamma}
 &\leq -a_{k}|y-x|^{\gamma}= -a_{k}|x-y|^{\gamma}\leq -a_{k}(|x_{1}-y_{1}|+|x_{2}-y_{2}|)^{\gamma}\\
 &\leq -a_{k}(1+|x_{1}-y_{1}|+|x_{2}-y_{2}|)
 \end{align*}
 because $(a_{n})_{n\in\N}$ is non-positive and increasing and $0<\gamma\leq 1$. We deduce that
 \begin{flalign*}
 &\hspace{0.37cm}\int_{\Omega_{I_{4}(n)}}\frac{|g_{n}(x-y)|\nu_{p}(x)}{|x-y|\nu_{k}(y)}\d y\\
 &=\int_{\Omega_{2n}}\frac{e^{-\re((x-y)^{2})}}{|x-y|}e^{a_{p}|x|^{\gamma}-a_{k}|y|^{\gamma}}\d y
  \leq \int_{\Omega_{2n}}\frac{e^{-\re((x-y)^{2})}}{|x-y|}e^{-a_{k}|x-y|^{\gamma}}\d y\\
 &\leq \int_{0}^{2\pi}\int_{0}^{1}\frac{e^{-r^{2}\cos(2\varphi)}}{r}e^{-a_{k}r^{\gamma}}r\d r\d\varphi
  + \int_{\Omega_{2n}\setminus \mathbb{B}_{1}(x)}e^{-\re((x-y)^{2})}e^{-a_{k}|x-y|^{\gamma}}\d y\\
 &\leq 2\pi e^{1-a_{k}}+e^{-a_{k}}\int_{-2n}^{2n}e^{(x_{2}-y_{2})^{2}-a_{k}|x_{2}-y_{2}|}\d y_{2}
        \int_{\R}e^{-(x_{1}-y_{1})^{2}-a_{k}|x_{1}-y_{1}|}\d y_{1}\\
 &\leq 2\pi e^{1-a_{k}}+4ne^{-a_{k}+(|x_{2}|+2n)^{2}-a_{k}(|x_{2}|+2n)}
    \int_{\R}e^{-y_{1}^{2}-a_{k}|y_{1}|}\d y_{1}\\
 &=2\pi e^{1-a_{k}}+4ne^{-a_{k}+(|x_{2}|+2n)^{2}-a_{k}(|x_{2}|+2n)}
    \int_{\R}e^{-(|y_{1}|+a_{k}/2)^{2}+a_{k}^{2}/4}\d y_{1}\\
 &=2\pi e^{1-a_{k}}+8ne^{-a_{k}+(|x_{2}|+2n)^{2}-a_{k}(|x_{2}|+2n)+a_{k}^{2}/4}
    \int_{a_{k}/2}^{\infty}e^{-y_{1}^{2}}\d y_{1}\\
 &\leq 2\pi e^{1-a_{k}}+8\sqrt{\pi}ne^{-a_{k}+(|x_{2}|+2n)^{2}-a_{k}(|x_{2}|+2n)+a_{k}^{2}/4}\\
 &\leq 2\pi e^{1-a_{k}}+8\sqrt{\pi}ne^{-a_{k}+(p+2n)^{2}-a_{k}(p+2n)+a_{k}^{2}/4}\\
 &\leq 2\pi e^{1-a_{I_{4}(n)}}+8\sqrt{\pi}ne^{-a_{I_{4}(n)}+(I_{14}(n)+2n)^{2}-a_{I_{4}(n)}(I_{14}(n)+2n)+a_{I_{4}(n)}^{2}/4}
\end{flalign*}
for $(k,p)=(I_{4}(n),n)$ and $(k,p)=(I_{14}(n),I_{14}(n))$ as $(-a_{n})_{n\in\N}$ is non-negative and decreasing.

$(WR.3)$: Let $M\subset\overline{\Omega}_{n}$ be closed and $N$ a component of $M^{C}$ 
 such that $N\cap\overline{\Omega}_{n}^{C}\neq\varnothing$. 
 We claim that $N\cap X_{I_{214}(n)}=N\cap\overline{\Omega}_{16n}^{C}\neq\varnothing$. 
 We note that $\overline{\Omega}_{16n}^{C}\subset\overline{\Omega}_{n}^{C}\subset M^{C}$ and 
 \begin{align*}
  \overline{\Omega}_{k}^{C}&=\phantom{\cup}\{z\in \C\;|\;\im(z)>k\}\cup\{z\in \C\;|\;\im(z)<-k\} \\
  &\phantom{=}\cup\{z\in\C\;|\;\d^{|\cdot|}(\{z\},\partial \Omega)<1/k\} 
  =:S_{1,k}\cup S_{2,k}\cup S_{3,k},\quad k\in\N.
 \end{align*}
 If there is $x\in N\cap\overline{\Omega}_{n}^{C}$ with $\im(x)>n$ or $\im(x)<-n$, then 
 $S_{1,16n}\subset S_{1,n}\subset N$ or $S_{2,16n}\subset S_{2,n}\subset N$ since $S_{1,n}$ and $S_{2,n}$ are connected 
 and $N$ a component of $M^{C}$. 
 If there is $x\in N\cap\overline{\Omega}_{n}^{C}$ such that $x\in S_{3,n}$, then there is $y\in\partial\Omega$ with 
 $x\in \mathbb{B}_{1/n}(y)\subset S_{3,n}$. This implies $\mathbb{B}_{1/(16n)}(y)\subset\mathbb{B}_{1/n}(y)\subset N$ 
 as $\mathbb{B}_{1/n}(y)$ is connected and $N$ a component of $M^{C}$, proving our claim.

 \item [b)] $(WR.1a)$: The choice $K:=\overline{\Omega}_{n}$ guarantees that this condition is fulfilled.
 
 $(WR.1b)$: We have $\d_{X,I_{2}}\geq 1/(2n)$ if $\partial\Omega\neq\varnothing$ 
 and $\d_{X,I_{2}}= 2n$ if $\Omega=\C$ for $(\Omega_{n})_{n\in\N}$ from (ii). 
 We choose $g_{n}\colon\C\to\C$, $g_{n}(z):=1$, as well as $r_{n}:=1/(4n)$ and $R_{n}:=1/(6n)$ for $n\in\N$. 
 Let $z\in\Omega_{I_{2}(n)}$ and $x\in X_{I_{2}(n)}+\mathbb{B}_{R_{n}}(0)$. For $\zeta=\zeta_{1}+i\zeta_{2}\in\C$ 
 with $|\zeta-(z-x)|=r_{n}$ we have $|g_{n}(\zeta)|\nu_{I_{2}(n)}(z)=1=:A_{2}(x,n)$.
 
 $(WR.1c)$: Let $K\subset\C$ be compact and $x=x_{1}+ix_{2}\in\Omega_{n}$. 
 Again, it follows from polar coordinates and Fubini's theorem that
 \begin{flalign*}
 &\hspace{0.37cm}\int_{K}\frac{|g_{n}(x-y)|}{|x-y|}\d y\\
 &=\int_{K}\frac{1}{|x-y|}\d y
  \leq \int_{\mathbb{B}_{1}(x)}\frac{1}{|x-y|}\d y
  +\int_{K\setminus \mathbb{B}_{1}(x)}\frac{1}{|x-y|}\d y\\
 &\leq \int_{0}^{2\pi}\int_{0}^{1}\frac{1}{r}r\d r\d\varphi
  + \int_{K\setminus \mathbb{B}_{1}(x)}1\d y\\
 &\leq 2\pi +\lambda(K),
 \end{flalign*}
 yielding $(WR.1c)$ because $\nu_{n}=1$. 
 
 $(WR.2)$: Follows from $(WR.1c)$.
 
 $(WR.3)$: Let $M\subset\overline{\Omega}_{n}$ be closed and $N$ a component of $M^{C}$ 
 such that $N\cap\overline{\Omega}_{n}^{C}\neq\varnothing$. 
 We claim that $N\cap X_{I_{214}(n)}=N\cap\overline{\Omega}_{16n}^{C}\neq\varnothing$. 
 We note that $\overline{\Omega}_{16n}^{C}\subset\overline{\Omega}_{n}^{C}\subset M^{C}$ and 
 \begin{align*}
  \overline{\Omega}_{k}^{C}&=\{z\in \C\;|\;|z|>k\}\cup\{z\in\C\;|\;\d^{|\cdot|}(\{z\},\partial \Omega)<1/k\}\\ 
  &=:S_{1,k}\cup S_{2,k},\quad k\in\N.
 \end{align*}
 If there is $x\in N\cap\overline{\Omega}_{n}^{C}$ with $|x|>n$, then 
 $S_{1,16n}\subset S_{1,n}\subset N$ since $S_{1,n}$ is connected 
 and $N$ a component of $M^{C}$. 
 If there is $x\in N\cap\overline{\Omega}_{n}^{C}$ such that $x\in S_{2,n}$, then there is $y\in\partial\Omega$ with 
 $x\in \mathbb{B}_{1/n}(y)\subset S_{2,n}$. This implies $\mathbb{B}_{1/(16n)}(y)\subset\mathbb{B}_{1/n}(y)\subset N$ 
 as $\mathbb{B}_{1/n}(y)$ is connected and $N$ a component of $M^{C}$, proving our claim.
\end{enumerate}
\end{proof}

Due to \prettyref{ex:families_of_weights_2} b) we get \cite[Theorem 1.4.4, p.\ 12]{H3} back.
For certain non-metrisable spaces $E$ the surjectivity of the Cauchy-Riemann operator 
in \prettyref{ex:families_of_weights_2} a) 
for $a_{n}=-1/n$, $n\in\N$, $\partial\Omega\subset \R$ and $\gamma=1$ 
is proved in \cite[5.24 Theorem, p.\ 95]{ich} by 
using the splitting theory of Vogt \cite{V1} and of Bonet and Doma\'nski \cite{Dom1} and that 
$\mathcal{EV}_{\overline{\partial}}(\Omega)$ has property $(\Omega)$ (see \cite[Definition, p.\ 367]{meisevogt1997}) 
in this case by \cite[5.20 Theorem, p.\ 84]{ich} and \cite[5.22 Theorem, p.\ 92]{ich}. This is generalised in \cite{kruse2019_1}. 

\subsection*{Acknowledgements}
This work is a generalisation of a part of the author’s Ph.D thesis \cite{ich}, written under
the advice of M. Langenbruch. The author would like to express his utmost
gratitude to him. Further, it is worth to mention that some of the results appearing in the Ph.D thesis 
and thus their generalised counterparts in the present paper are essentially due to him.

%% file: dquer_main.bbl
\begin{thebibliography}{30}
\providecommand{\natexlab}[1]{#1}
\providecommand{\url}[1]{\texttt{#1}}
\expandafter\ifx\csname urlstyle\endcsname\relax
  \providecommand{\doi}[1]{doi: #1}\else
  \providecommand{\doi}{doi: \begingroup \urlstyle{rm}\Url}\fi

\bibitem[Amar(2016)]{Amar2016}
E.~Amar.
\newblock {Serre duality and H\"ormander's solution of the
  $\overline{\partial}$-equation}.
\newblock \emph{Bull. Sci. Math.}, 140\penalty0 (6):\penalty0 747--756, 2016.
\newblock \doi{10.1016/j.bulsci.2016.01.004}.

\bibitem[Bierstedt et~al.(1975)Bierstedt, Gramsch, and Meise]{Bierstedt1975}
K.-D. Bierstedt, B.~Gramsch, and R.~Meise.
\newblock {Lokalkonvexe Garben und gewichtete induktive Limites
  $\mathfrak{F}$-morpher Funktionen}.
\newblock In J.~Blatter, J.~B. Prolla, and W.~Rue{\ss}, editors, \emph{Function
  spaces and dense approximation (Proc., Bonn, 1974)}, Bonner Math. Schriften
  81, pages 59--72, Bonn, 1975. Inst. Angew. Math., Univ. Bonn.

\bibitem[Bonami and Charpentier(1990)]{Bonami1990}
A.~Bonami and P.~Charpentier.
\newblock {Boundary values for the canonical solution to
  $\overline{\partial}$-equation and $W^{1/2}$ estimates}, 1990.
\newblock preprint
  \url{https://www.math.u-bordeaux.fr/~pcharpen/recherche/data/Bon-Ch-Neumann.pdf}
  (24 October 2020).

\bibitem[Bonet and Doma\'nski(2008)]{Dom1}
J.~Bonet and P.~Doma\'nski.
\newblock {The splitting of exact sequences of PLS-spaces and smooth dependence
  of solutions of linear partial differential equations}.
\newblock \emph{Adv. Math.}, 217:\penalty0 561--585, 2008.
\newblock \doi{10.1016/j.aim.2007.07.010}.

\bibitem[Charpentier et~al.(2014)Charpentier, Dupain, and
  Mounkaila]{Charpentier2014}
P.~Charpentier, Y.~Dupain, and M.~Mounkaila.
\newblock {Estimates for solutions of the $\overline{\partial}$-equation and
  application to the characterization of the zero varieties of the functions of
  the Nevanlinna class for lineally convex domains of finite type}.
\newblock \emph{J. Geom. Anal.}, 24\penalty0 (4):\penalty0 1860--1881, 2014.
\newblock \doi{10.1007/s12220-013-9398-5}.

\bibitem[Defant and Floret(1993)]{Defant}
A.~Defant and K.~Floret.
\newblock \emph{Tensor norms and operator ideals}.
\newblock Math. Stud. 176. North-Holland, Amsterdam, 1993.

\bibitem[Epifanov(1992)]{Epifanov1992}
O.~V. Epifanov.
\newblock {On solvability of the nonhomogeneous Cauchy-Riemann equation in
  classes of functions that are bounded with weights or systems of weights}.
\newblock \emph{Math. Notes}, 51\penalty0 (1):\penalty0 54--60, 1992.
\newblock \doi{10.1007/BF01229435}.

\bibitem[Floret and Wloka(1968)]{F/W/Buch}
K.~Floret and J.~Wloka.
\newblock \emph{Einf\"uhrung in die Theorie der lokalkonvexen R\"aume}.
\newblock Lecture Notes in Math. 56. Springer, Berlin, 1968.
\newblock \doi{10.1007/BFb0098549}.

\bibitem[Haslinger(2001)]{Haslinger2001}
F.~Haslinger.
\newblock {The canonical solution operator to $\overline{\partial}$ restricted
  to Bergman spaces}.
\newblock \emph{{Proc. Am. Math. Soc.}}, 129\penalty0 (11):\penalty0
  3321--3329, 2001.
\newblock \doi{10.1090/S0002-9939-01-05953-6}.

\bibitem[Haslinger(2002)]{Haslinger2002}
F.~Haslinger.
\newblock {The canonical solution operator to $\overline{\partial}$ restricted
  to spaces of entire functions}.
\newblock \emph{{Ann. Fac. Sci. Toulouse, Math. (6)}}, 11\penalty0
  (1):\penalty0 57--70, 2002.
\newblock \doi{10.5802/afst.1018}.

\bibitem[Haslinger and Helffer(2007)]{Haslinger2007}
F.~Haslinger and B.~Helffer.
\newblock {Compactness of the solution operator to $\overline{\partial}$ in
  weighted $L^{2}$-spaces}.
\newblock \emph{J. Funct. Anal.}, 243\penalty0 (2):\penalty0 679--697, 2007.
\newblock \doi{10.1016/j.jfa.2006.09.004}.

\bibitem[Hedenmalm(2015)]{Hedenmalm2015}
H.~Hedenmalm.
\newblock {On H\"ormander's solution of the $\overline{\partial}$-equation. I}.
\newblock \emph{Math. Z.}, 281\penalty0 (1-2):\penalty0 349--355, 2015.
\newblock \doi{10.1007/s00209-015-1487-7}.

\bibitem[H\"ormander(1990)]{H3}
L.~H\"ormander.
\newblock \emph{An introduction to complex analysis in several variables}.
\newblock North-Holland, Amsterdam, 3rd edition, 1990.

\bibitem[H\"ormander(2003)]{H1}
L.~H\"ormander.
\newblock \emph{The Analysis of linear partial differential operators I}.
\newblock Classics Math. Springer, Berlin, 2nd edition, 2003.
\newblock \doi{10.1007/978-3-642-61497-2}.

\bibitem[H{\"o}rmander(2003)]{Hoermander2003}
L.~H{\"o}rmander.
\newblock {A history of existence theorems for the Cauchy-Riemann complex in
  $L^2$ spaces}.
\newblock \emph{J. Geom. Anal.}, 13\penalty0 (2):\penalty0 329--357, 2003.
\newblock \doi{10.1007/BF02930700}.

\bibitem[H\"ormander(2005)]{H2}
L.~H\"ormander.
\newblock \emph{The Analysis of linear partial differential operators II}.
\newblock Classics Math. Springer, Berlin, 2nd edition, 2005.
\newblock \doi{10.1007/b138375}.

\bibitem[Jarchow(1981)]{Jarchow}
H.~Jarchow.
\newblock \emph{Locally Convex Spaces}.
\newblock Math. Leitf\"{a}den. Teubner, Stuttgart, 1981.
\newblock \doi{10.1007/978-3-322-90559-8}.

\bibitem[Kaballo(2014)]{Kaballo}
W.~Kaballo.
\newblock \emph{Aufbaukurs Funktionalanalysis und Operatortheorie}.
\newblock Springer, Berlin, 2014.
\newblock \doi{10.1007/978-3-642-37794-5}.

\bibitem[Kruse(2014)]{ich}
K.~Kruse.
\newblock \emph{Vector-valued Fourier hyperfunctions}.
\newblock PhD thesis, Universit\"at Oldenburg, Oldenburg, 2014.

\bibitem[Kruse(2019)]{kruse2018_2}
K.~Kruse.
\newblock The approximation property for weighted spaces of differentiable
  functions.
\newblock In M.~Kosek, editor, \emph{{Function Spaces XII (Proc., Krak\'ow,
  2018)}}, volume 119 of \emph{Banach Center Publ.}, pages 233--258, Warszawa,
  2019. Inst. Math., Polish Acad. Sci.
\newblock \doi{10.4064/bc119-14}.

\bibitem[Kruse(2020{\natexlab{a}})]{kruse2017}
K.~Kruse.
\newblock {Weighted spaces of vector-valued functions and the
  $\varepsilon$-product}.
\newblock \emph{Banach J. Math. Anal.}, 14\penalty0 (4):\penalty0 1509--1531,
  2020{\natexlab{a}}.
\newblock \doi{10.1007/s43037-020-00072-z}.

\bibitem[Kruse(2020{\natexlab{b}})]{kruse2018_4}
K.~Kruse.
\newblock On the nuclearity of weighted spaces of smooth functions.
\newblock \emph{Ann. Polon. Math.}, 124\penalty0 (2):\penalty0 173--196,
  2020{\natexlab{b}}.
\newblock \doi{10.4064/ap190728-17-11}.

\bibitem[Kruse(2020{\natexlab{c}})]{kruse2019_1}
K.~Kruse.
\newblock Parameter dependence of solutions of the {C}auchy-{R}iemann equation
  on weighted spaces of smooth functions.
\newblock \emph{RACSAM}, 114\penalty0 (3):\penalty0 1--24, 2020{\natexlab{c}}.
\newblock \doi{10.1007/s13398-020-00863-x}.

\bibitem[Langenbruch(1992)]{L4}
M.~Langenbruch.
\newblock {Splitting of the $\overline{\partial}$-complex in weighted spaces of
  square integrable functions}.
\newblock \emph{Revista Math. Complut.}, 5\penalty0 (2-3):\penalty0 201--223,
  1992.

\bibitem[Langenbruch(New York)]{Langenbruch1994}
M.~Langenbruch.
\newblock Differentiable functions and the $\overline{\partial}$-complex.
\newblock In K.-D. Bierstedt, A.~Pietsch, W.~M. Ruess, and D.~Vogt, editors,
  \emph{Functional Analysis, Proceedings of the Essen Conference}, Lect. Notes
  in Pure and Appl. Math. 150, pages 415--434, Dekker, New York. 1994.

\bibitem[Meise and Vogt(1997)]{meisevogt1997}
R.~Meise and D.~Vogt.
\newblock \emph{Introduction to Functional Analysis}.
\newblock Oxf. Grad. Texts Math. 2. Clarendon Press, Oxford, 1997.

\bibitem[Polyakova(2017)]{Polyakova2017}
D.~A. Polyakova.
\newblock {Solvability of the inhomogeneous Cauchy-Riemann equation in
  projective weighted spaces}.
\newblock \emph{Siberian Math. J.}, 58\penalty0 (1):\penalty0 142--152, 2017.
\newblock \doi{10.1134/S0037446617010189}.

\bibitem[Tr\`{e}ves(2006)]{Treves}
F.~Tr\`{e}ves.
\newblock \emph{Topological Vector Spaces, Distributions and Kernels}.
\newblock Dover, Mineola, NY, 2006.

\bibitem[Vogt(1987)]{V1}
D.~Vogt.
\newblock {On the functors $\operatorname{Ext}^{1}\left(E,F\right)$ for
  Fr\'{e}chet spaces}.
\newblock \emph{Studia Math.}, 85(2):\penalty0 163--197, 1987.
\newblock \doi{10.4064/sm-85-2-163-197}.

\bibitem[Wloka(1967)]{Wloka1967}
J.~Wloka.
\newblock {Reproduzierende Kerne und nukleare R{\"a}ume. II}.
\newblock \emph{Math. Ann.}, 172\penalty0 (2):\penalty0 79--93, 1967.
\newblock \doi{10.1007/BF01350088}.

\end{thebibliography}
